\documentclass[11pt]{amsart}
\usepackage{amsmath,amsfonts,amsthm,amscd,amssymb,graphicx}

\newcommand{\ip}{\int_\mathbb{R}}

\newtheorem{theorem}{Theorem}[section]
\newtheorem{lemma}[theorem]{Lemma}

\newtheorem*{remark}{Remark}

\def\CalO{\mathcal{O}}

\def\bV{{\hat{v}}}
\def\R{\Re e}
\def\I{\Im m}

\begin{document}
\title[Spectrum for Isentropic Gas Dynamics with Capillarity]{On the Shock Wave Spectrum for Isentropic Gas Dynamics with Capillarity}
\author{Jeffrey Humpherys}
\address{Department of Mathematics, Brigham Young University, Provo, UT 84602}
\email{jeffh at math.byu.edu}

\subjclass{35Q30, 35Q35}
\date{\today}
\keywords{gas dynamics, shock wave stability, capillarity}
\thanks{The author thanks Kevin Zumbrun for useful conversations throughout this project.  This work was supported in part by the National Science Foundation, award number DMS-0607721.}

\begin{abstract}
We consider the stability problem for shock layers in Slemrod's model of an isentropic gas with capillarity.  We show that these traveling waves are monotone in the weak capillarity case, and become highly oscillatory as the capillarity strength increases.  Using a spectral energy estimate we prove that small-amplitude monotone shocks are spectrally stable.  We also show, through the use of a novel spectral energy estimate, that monotone shocks have no unstable real spectrum regardless of amplitude; this implies that any instabilities of these monotone traveling waves, if they exist, must occur through a Hopf-like bifurcation, where one or more conjugate pairs of eigenvalues cross the imaginary axis.  We then conduct a systematic numerical Evans function study, which shows that monotone and mildly oscillatory profiles in an adiabatic gas are spectrally stable for moderate values of shock and capillarity strengths.  In particular, we show that the transition from monotone to non-monotone profiles does not appear to trigger any instabilities.
\end{abstract}
\maketitle
\thispagestyle{empty}
\section{Introduction}
\label{int}
We consider Slemrod's model \cite{HaSl,Sl.1,Sl.2} for a one-dimensional isentropic gas with capillarity
\begin{equation}
\begin{split}
\label{slemrod}
v_{t}-u_{x} &=0, \\
u_{t}+p(v)_{x} &= \left(\frac{u_x}{v}\right)_x - d v_{xxx},
\end{split}
\end{equation}
where physically, $v$ is the specific volume, $u$ is the velocity in Lagrangian coordinates, $p(v)$ is the pressure law for an ideal gas, that is $p'(v) < 0$ and $p''(v) > 0$, and the coefficient, $d \geq 0$, accounting for capillarity strength, is constant.  This is model is a generalization of the compressible isentropic Navier-Stokes equations, or the $p$-system with semi-parabolic (or real) viscosity,
\begin{equation}
\begin{split}
\label{int:navierstokes}
v_{t}-u_{x} &=0, \\
u_{t}+p(v)_{x} &= \left(\frac{u_x}{v}\right)_x.
\end{split}
\end{equation}
It has recently been shown that viscous shock wave solutions of \eqref{int:navierstokes} are spectrally stable for all amplitudes in the case of an adiabatic gas law $p(v) = v^{-\gamma}$, $\gamma\in[1,3]$; see \cite{BHRZ,HLZ}.  We remark that this result, together with Mascia \& Zumbrun's work  \cite{MZ.2, MZ.3} implies that viscous shocks are asymptotically orbitally stable (hereafter called nonlinearly stable).   In this paper, we make the first step toward generalizing this work to Slemrod's model by showing that monotone and mildly oscillatory smooth shock profiles of small to moderate amplitude are likewise spectrally stable.

More generally, we are interested in understanding the degree to which the analytic methods used to study shock wave stability in viscous conservation laws extend to viscous-dispersive systems.  We view Slemrod's model as an important test case as it is physically realistic and yet captures some of the essential mathematical hurdles found in more extensive models of compressible fluid flow.  In particular, Slemrod's model is symmetrizable and genuinely coupled, having only semi-parabolic diffusion; see \cite{Hu.2} for details. 

A few notable results in the study of shock wave stability for viscous conservation laws include the works of Kawashima \cite{Ka.1,Ka.2,KS}, who proved that genuinely coupled symmetrizable systems has a stable essential spectrum, the works of Goodman and others \cite{Go.1,MN, KMN, HuZ.1}, who proved small-amplitude spectral stability for viscous shocks through the use of cleverly chosen weighted energy estimates, and the works of Zumbrun and collaborators \cite{ZH,MZ.1,MZ.2, MZ.3, Z.2}, who proved that spectral stability implies nonlinear stability for viscous shocks in conservation laws for both strictly parabolic and semi-parabolic viscosities.  The missing piece in this overall program is to determine whether moderate- and large-amplitude viscous shocks are spectrally stable.  Very recently, however, spectral stability for large-amplitude shocks for \eqref{int:navierstokes} was proven in the case of an adiabatic gas \cite{HLZ}, and spectral stability was numerically demonstrated for the intermediate range through an extensive Evans function study \cite{BHRZ}.  There is some hope that this overall strategy will extend to more general systems of viscous conservation laws and perhaps even viscous-dispersive models.

We remark that Kawashima's admissibility results, mentioned above, were recently extended to viscous-dispersive (and higher-order) systems \cite{Hu.2}.  Also, Howard \& Zumbrun showed that spectral stability implies nonlinear stability for scalar viscous-dispersive conservation laws \cite{HZ}.   However, the remaining pieces of the general program for viscous-dispersive systems, described above, are still open.

This paper is organized as follows:  In Section \ref{prelim}, we set the stage by first proving the existence of shock profiles for \eqref{slemrod} through the use of a Lyapunov function argument.  Then using geometric singular perturbation theory, we show that small-amplitude shock profiles converge to the zero-capillarity case and are thus monotone.  Following that, we use a qualitative ODE argument to show that our profiles are monotone for weak capillarity yet become highly oscillatory as the capillarity strength $d$ increases.   We then provide a short estimate on the derivative bounds of the profile, which are used later in the stability analysis. Finally, we formulate the integrated eigenvalue problem, which makes the stability problem more amenable to analysis; see for example \cite{Go.1,ZH}.  In Section  \ref{spec}, we generalize the work of Matsumura \& Nishihara \cite{MN} and Barker, Humpherys, Rudd, \& Zumbrun \cite{BHRZ} by using a spectral energy estimate to prove that small-amplitude monotone shocks of \eqref{slemrod} are spectrally stable.  In Section \ref{noreal}, we further extend the results in \cite{BHRZ} and offer a short and novel proof that monotone shocks have no unstable real spectrum regardless of amplitude.  This restricts the class of admissible bifurcations for monotone profiles to those of Hopf-type, where one or more conjugate pairs of eigenvalues cross the imaginary axis.  The approach used here is different than many energy methods in that we use a spectral energy estimate that does not appear to have a time-asymptotic equivalent, whereas most energy estimates can be performed in either domain.  In Section \ref{hfb}, we extend the spectral bounds in \cite{BHRZ} to \eqref{slemrod} by proving that high-frequency instabilities cannot occur for adiabatic monotone profiles of any amplitude for $d\leq 1/3$.  Finally in Section \ref{evans}, we carry out a systematic numerical Evans function study showing that adiabatic monotone and mildly oscillatory profiles are spectrally stable for moderate shock and capillarity strengths.

We remark that highly oscillatory profiles in the scalar KDV-Burgers model were shown by Pego, Smerka, \& Weinstein \cite{PSW} to be unstable in certain cases.  Thus for some, perhaps extreme, parameters, one can reasonably expect instabilities to occur in our system as well.  It is challenging, however, with current technology to explore these extreme cases numerically.  We plan on exploring this in the future.

\section{Preliminaries}
\label{prelim}

In this section, we derive the profile ODE and provide a convenient scaling for our analysis.  We prove the existence of shock profiles for \eqref{slemrod} through the use of a Lyapunov function argument.  Then using geometric singular perturbation theory, we show that small-amplitude shock profiles converge to the zero-capillarity case and are thus monotone.  Through a qualitative ODE argument, we then show that profiles are monotone for weak capillarity yet become oscillatory as the capillarity strength $d$ increases beyond the transition point $d_*$.  We then provide a short estimate on the derivative bounds of the profile, which will be used later in the stability analysis. Finally, we formulate the spectral stability problem and change to integrated coordinates making it more amenable to analysis; see for example \cite{Go.1,ZH}.

\subsection{Shock Profiles}
\label{scalingsection}
By a shock layer (or shock profile) of \eqref{slemrod}, we mean a traveling wave solution
\begin{equation*}
\begin{split}
v(x,t)&=\hat{v}(x-s t),\\
u(x,t)&=\hat{u}(x-s t),
\end{split}
\end{equation*}
with asymptotically constant end-states $(\hat{v},\hat{u})(\pm \infty) =(v_{\pm},u_{\pm})$. Rather by translating $x \rightarrow x-s t$, we can instead consider  stationary solutions of
\[
\begin{split}
v_{t} - s v_{x} - u_{x} &= 0,\\
u_{t} - s u_{x} + p(v)_{x} &= \left(\frac{u_x}{v}\right)_x - d v_{xxx}.
\end{split}
\]
Under the rescaling 
$(x,t,u) \rightarrow (-sx, s^2 t, -u/s)$, our system takes the form
\begin{equation}
\begin{split}
\label{trav:cap2}
v_t + v_x - u_x &= 0,\\
u_t + u_x + a  p(v)_{x} &= \left(\frac{u_x}{v}\right)_x - d v_{xxx},
\end{split}
\end{equation}
where $a = 1/s^2$.  Thus, the shock profiles of \eqref{slemrod} are solutions of the ordinary differential equation
\begin{align*}
v' - u' &= 0,\\
u' + a  p(v)' &= \left(\frac{u'}{v}\right)' - d v''',
\end{align*}
subject to the boundary conditions $(v,u)(\pm \infty) =(v_{\pm},u_{\pm})$. This simplifies to
\begin{equation*}
v' + a  p(v)' = \left(\frac{v'}{v}\right)' - d v'''.
\end{equation*}
By integrating from $-\infty$ to $x$, we get our profile equation,
\begin{equation}
\label{profile}
v-v_{-} + a  (p(v)-p(v_{-})) = \frac{v'}{v} - d v'' ,
\end{equation}
where $a$ is found by setting $x=+\infty$, thus yielding the Rankine-Hugoniot condition
\begin{equation}
\label{RH}
a = -\frac{v_+ - v_-}{p(v_+) - p(v_-)}.
\end{equation}
Without loss of generality, we will assume that $0 < v_+ < v_-$.   We remark that small-amplitude shocks  occur when $v_+$ is close to $v_-$ and large-amplitude shocks arise to when $v_+$ nears zero.

\begin{remark}
\label{cap:rmk1}
In the absence of capillarity, that is when $d=0$, the profile equation \eqref{profile} is of first order, and thus has a monotone solution.  As we will show, small values of $d$ likewise yield monotone profiles whereas large values of $d$ produce oscillatory profiles.  We make this precise below.
\end{remark}

\subsection{Adiabatic Gas}
\label{adiabatic}
Although much of the analysis in this paper holds for ideal gases, that is when $p'(v)<0$ and $p''(v)>0$, our numerical study focuses on the special case of an adiabatic gas law,
\begin{equation}
\label{gammalaw}
p(v) = v^{-\gamma}, \gamma\geq 1,
\end{equation}
together with the rescaling
\[
(x,t,v,u,a,d) \rightarrow (\varepsilon x, \varepsilon t, v/\varepsilon, u/\varepsilon,a \varepsilon^{-\gamma-1},\varepsilon^2 d),
\]
where $\varepsilon$ is chosen so that $v_-=1$;  see \cite{BHRZ,HLZ} for more details.  This choice simplifies our analysis in Section \ref{hfb} and also gives the Mach number $M$ the simplifying form $M=1/\sqrt{\gamma a}$.

\subsection{Existence}
We prove existence of profiles by the following Lyapunov function argument.
By writing \eqref{profile} as a first order system, we get
\begin{subequations}
\label{firstorder}
\begin{align}
v' &= w, \label{firstorder1}\\
w' &= \frac{1}{d} \left[ \frac{w - \phi(v) }{v} \right], \label{firstorder2}
\end{align}
\end{subequations}
where
\begin{equation}
\label{phi}
\phi(v)= v(v-v_{-} + a  (p(v)-p(v_{-})).
\end{equation}
The zero-diffusion case is conservative and has a corresponding Hamiltonian that provides us with the needed Lyapunov function.  Specifically, let
\begin{equation}
E(v,w) = \frac{1}{2} w^2 - \frac{1}{d} \int^{v_{-}}_{v} \frac{\phi(\tilde v)}{\tilde v} d\tilde v.
\end{equation}
Since $\phi(v) < 0$ on $(v_+,v_-)$, then $E(v,w)$ is non-negative for $v \in [v_{+},v_{-}]$.  It follows that
\begin{equation}
\frac{d}{dx} E(v(x),w(x))= \nabla E \cdot
(v',w')^T=\frac{w^2}{d v} > 0.
\end{equation}
Hence with diffusion, bounded (homoclinic) orbits at $(v_+,0)$ are pulled into the minimum $(v_{-},0)$ of $E(v,w)$ as $x\rightarrow -\infty$.  Thus there exists a connecting orbit from $v_+$ to $v_-$.

\subsection{The Small-Amplitude Limit}
\label{smallamp}

We now show that small-amplitude shocks of \eqref{slemrod} are monotone and follow the same asymptotic limits as the $d=0$ case presented in \cite{MP,P,BHRZ,HLZ}.  We accomplish this by rescaling and showing, via geometric singular perturbation theory \cite{F,GS,J}, that the profile converges, in the small-amplitude shock limit, to the (monotone) non-dispersive case.  Thus, monotonicity of small-amplitude shocks of \eqref{slemrod} is implied by the monotonicity of the non-dispersive case, as mentioned above.

\begin{lemma}
Small-amplitude shocks of \eqref{slemrod} are monotone for any fixed $d$.
\end{lemma}

\begin{proof}
We scale according to the amplitude $\varepsilon = v_{-}-v_{+}$.  Let $\bar{v} = (v-v_{0})/\varepsilon$ and $\bar{x}  = \varepsilon x$, where $v_{0}=v_{-}-\varepsilon \bar{v}_-$.  This frame is chosen so that the end-states of the profile are fixed at $\bar{v}_{\pm} = \mp 1/2$.  Additionally, we expand the pressure term $p(v)$ and the viscosity term $v^{-1}$ about $v_{-}$. Hence \eqref{profile} becomes
\begin{equation}
\label{jxexp}
\begin{split}
&\varepsilon (\bar{v}-\bar{v}_{-}) \left( 1+ a  p(v_{-}) \right) + \varepsilon^2
\frac{a p''(v_{-})}{2}(\bar{v}-\bar{v}_{-})^2 +
\mathcal{O}(\varepsilon^3)(\bar{v}-\bar{v}_{-})^3 \\
& \qquad = \varepsilon^2 \frac{\bar{v}'}{\bar{v}_-} + \mathcal{O}(\varepsilon^3)
(\bar{v}-\bar{v}_{-}) \bar{v}' + \varepsilon^3 d \bar{v}''.
\end{split}
\end{equation}
By expanding the Rankine-Hugoniot equality, $\varepsilon = a   (p(v_{+}) - p(v_{-}))$, about $v_{-}$, we obtain
\begin{equation}
\label{jxexp1}
1+ a  p'(v_{-}) = \frac{a p''(v_{-})}{2}\varepsilon + \mathcal{O}(\varepsilon^2).
\end{equation}
Substituting \eqref{jxexp1} into \eqref{jxexp} and simplifying gives (recall that $\bar{v}_{-} = 1/2$)
\begin{equation}
\frac{a p''(v_{-})}{2}(\bar{v}^2 - \frac{1}{4}) + \varepsilon R(\bar{v},\bar{v}') = \frac{\bar{v}'}{\bar{v}_-} +
\varepsilon^3 d \bar{v}'' \label{jxexp2}.
\end{equation}
where $R(\bar{v},\bar{v}') = \mathcal{O}(1)$.  Thus, in the $\varepsilon=0$ limit,
\eqref{jxexp2} becomes
\begin{equation}
\bar{v}' = \frac{a p''(v_{-})v_-}{2}(\bar{v}^2 - \frac{1}{4})
\label{jxexp3},
\end{equation}
which is essentially the same reduction obtained for the viscous Burgers equation.  Note that the capillarity term vanishes as well, and thus the reduction is the same as the zero-capillarity $(d=0)$ case.

The slow dynamics of \eqref{jxexp2} take the form
\begin{subequations}\label{jxslow}
\begin{align}
\bar{v}' &= \bar{w}, \label{jxslow:1}\\
\varepsilon \bar{w}' &= \frac{1}{d} \left[ \frac{a p''(v_{-})}{2}(\bar{v}^2 - \frac{1}{4}) +
\varepsilon R(\bar{v},\bar{v}')- \frac{\bar{w}}{\bar{v}_-} \right]. \label{jxslow:2}
\end{align}
\end{subequations}
The fast dynamics, obtained by rescaling $x \rightarrow x/\varepsilon$, take the form
\begin{subequations}\label{jxfast}
\begin{align}
\bar{v}' &= \varepsilon \bar{w}, \label{jxfast:1}\\
\bar{w}' &= \frac{1}{d} \left[ \frac{a p''(v_{-})}{2}(\bar{v}^2 - \frac{1}{4}) + \varepsilon R(\bar{v},\bar{v}') - \frac{\bar{w}}{\bar{v}_-} \right]. \label{jxfast:2}
\end{align}
\end{subequations}
We can see from the slow dynamics that solutions will remain on the parabola defined by 
\begin{equation*}
\bar{w} = \frac{a p''(v_{-}) v_-}{2}(\bar{v}^2-\frac{1}{4}).\label{jxexp5}
\end{equation*}
In addition, we can see from the fast dynamics that any jumps will be vertical, that is, $v=$ constant.  Since there are no vertical branches, no jumps occur and thus it follows that small-amplitude shocks approach the solutions for \eqref{jxexp3}.  
Hence, for sufficiently small amplitudes, the profiles are monotone.
\end{proof}

\begin{remark}
In the original scale, small-amplitude profiles of \eqref{int:navierstokes} have the asymptotic properties $|\bV_x|=\CalO(\varepsilon^2)$ and $|\bV_{xx}| = |\bV_x| \CalO(\varepsilon)$, where $\varepsilon = v_--v_+$ is the amplitude; see \cite{MP,P}.  From the above argument, these asymptotic properties hold with our scaling in \eqref{trav:cap2} as well.  It is also straightforward to establish these asymptotic properties directly; see for example Theorem \ref{asympthm} below.
\end{remark}

\begin{figure}[p]
\begin{center}
$\begin{array}{cc}
\includegraphics[width=6cm]{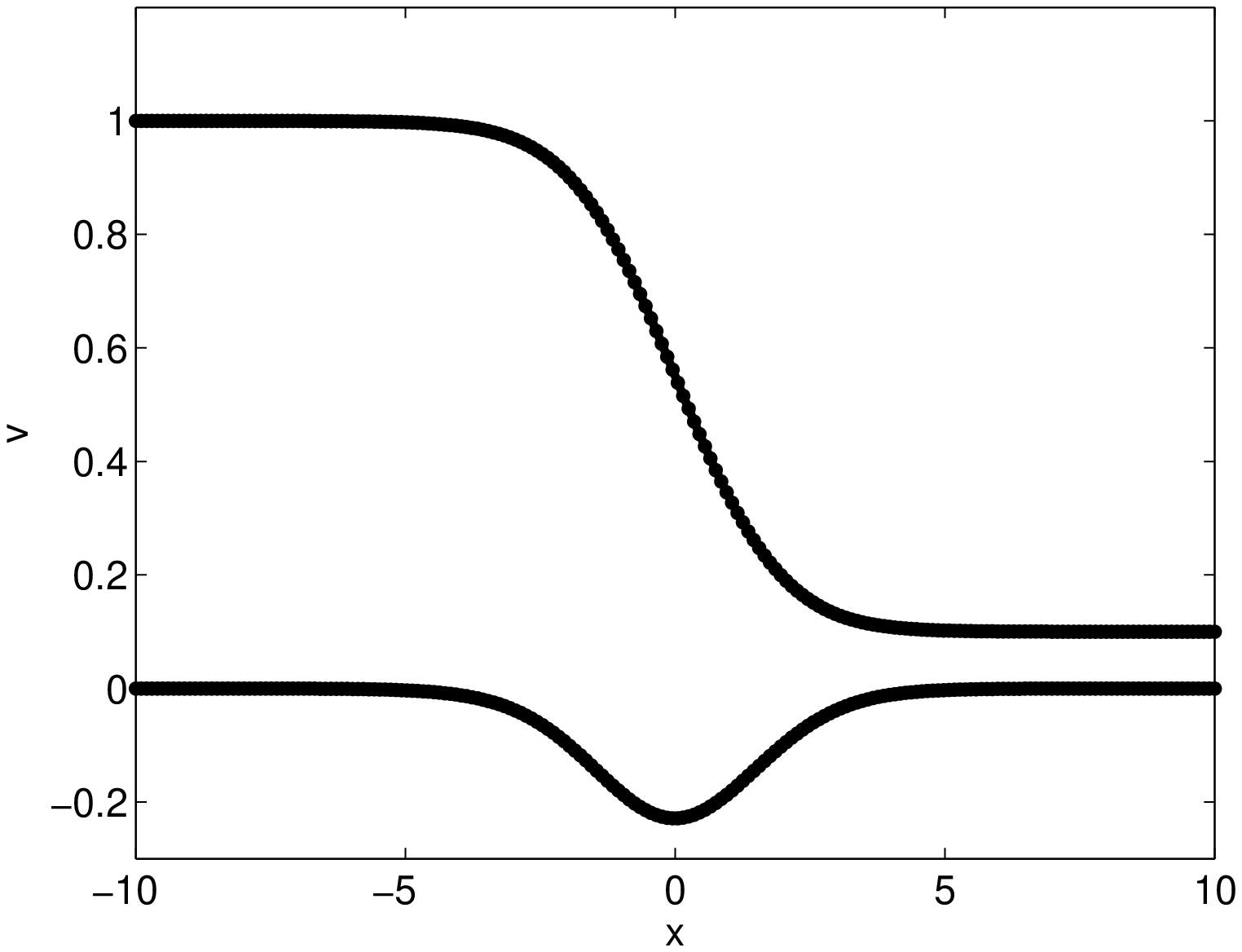} & \includegraphics[width=6cm]{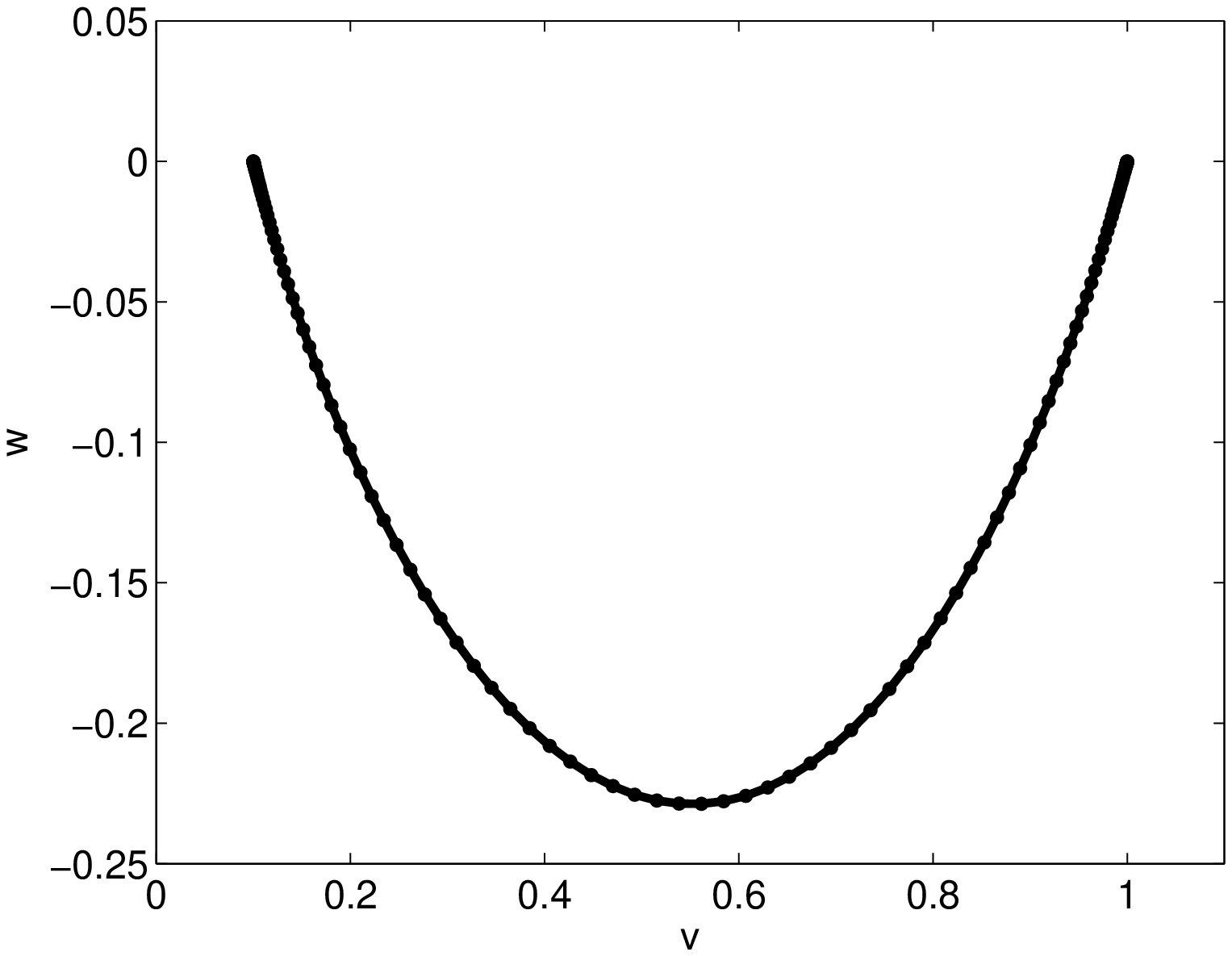} \\
\multicolumn{2}{c}{\mbox{\bf (a)}}\\
\includegraphics[width=6cm]{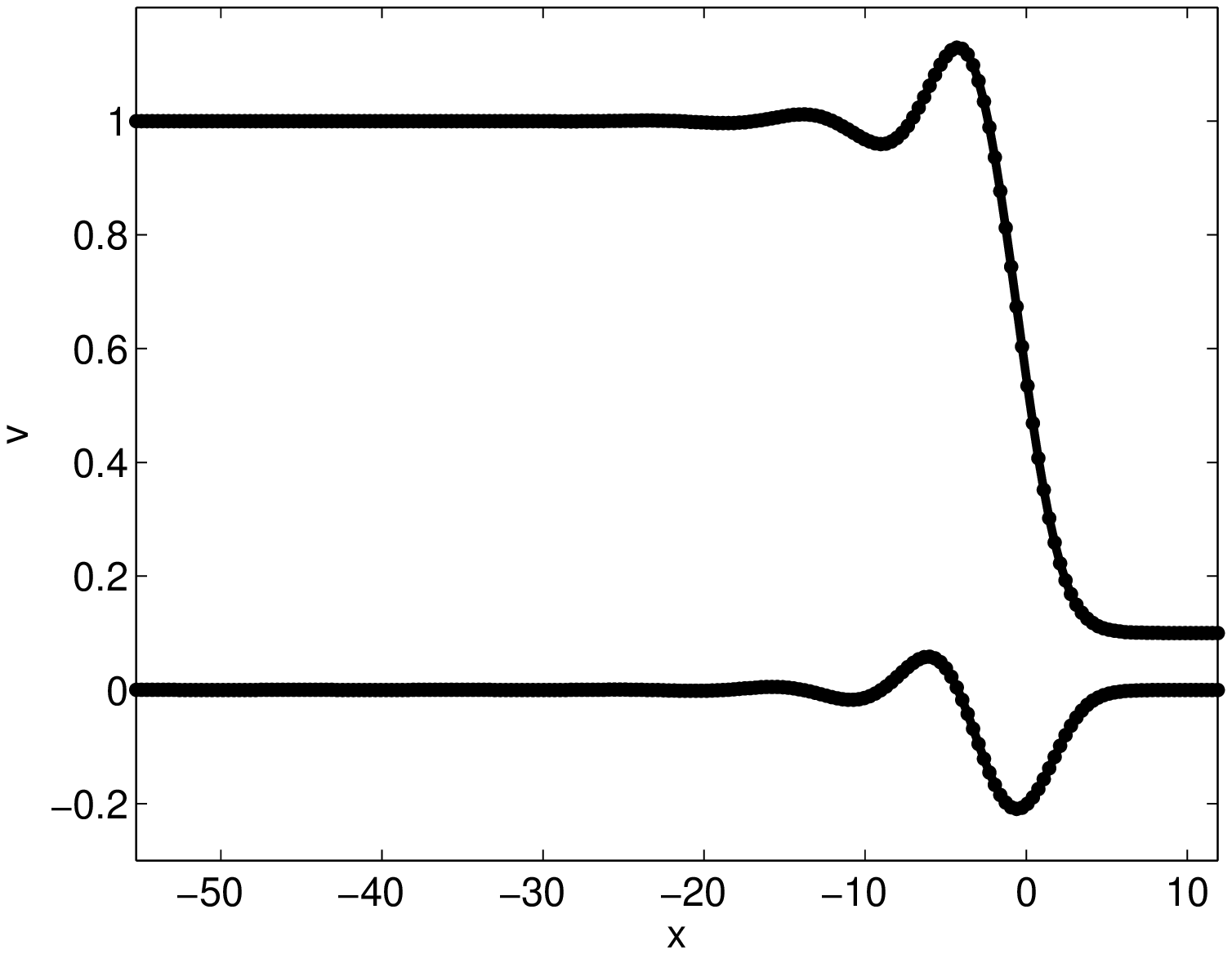} & \includegraphics[width=6cm]{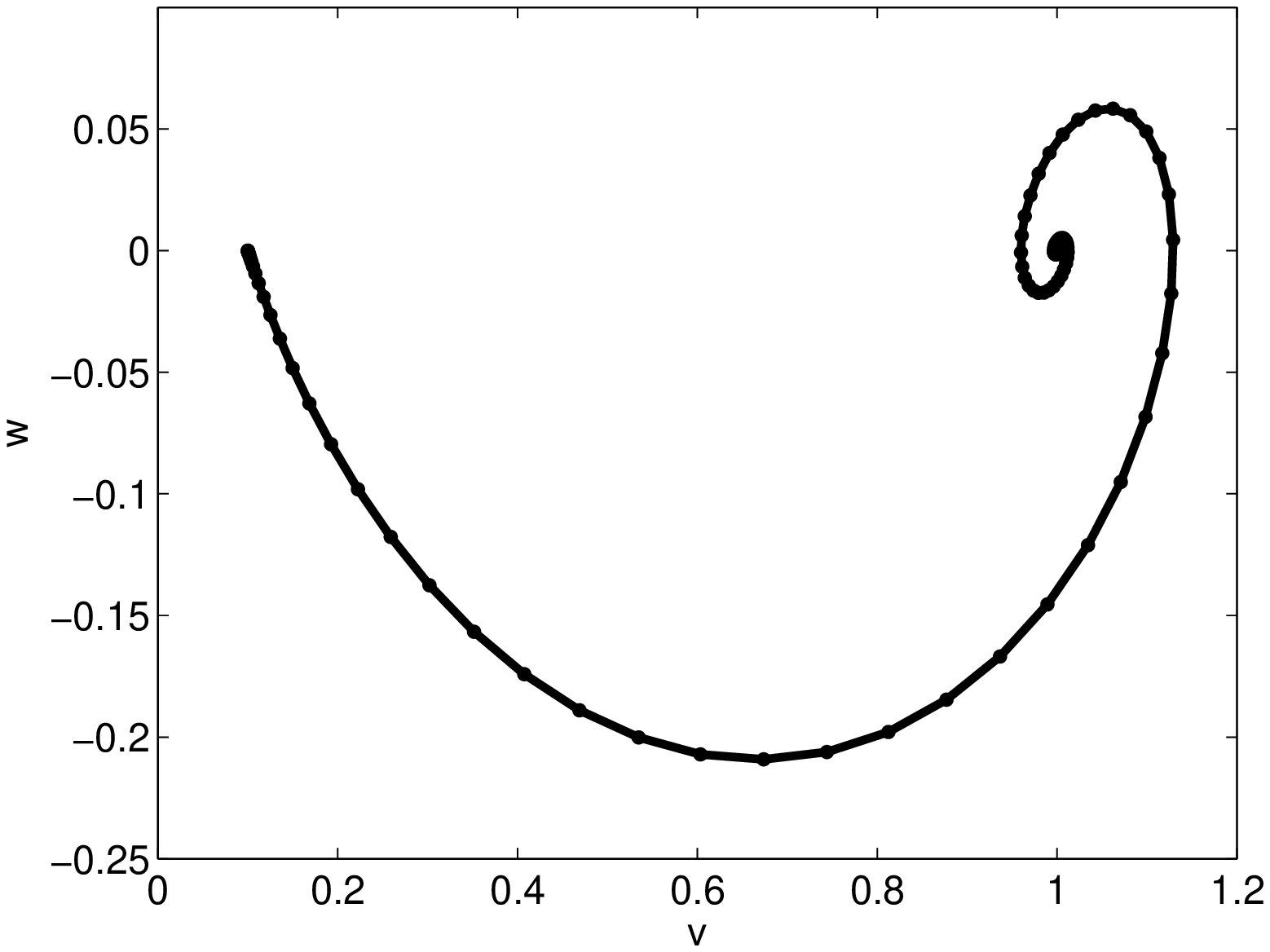} \\
\multicolumn{2}{c}{\mbox{\bf (b)}}\\
\includegraphics[width=6cm]{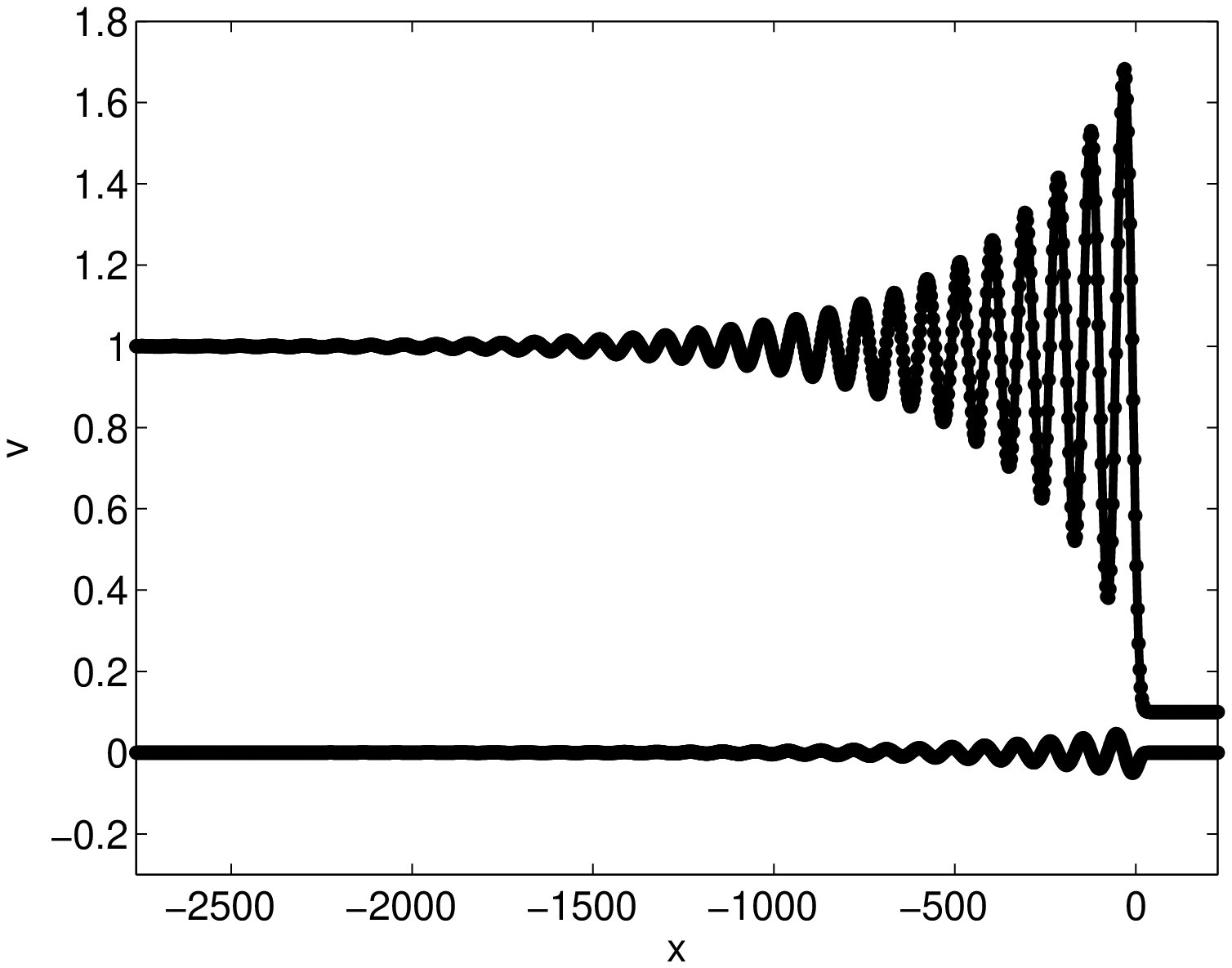} & \includegraphics[width=6cm]{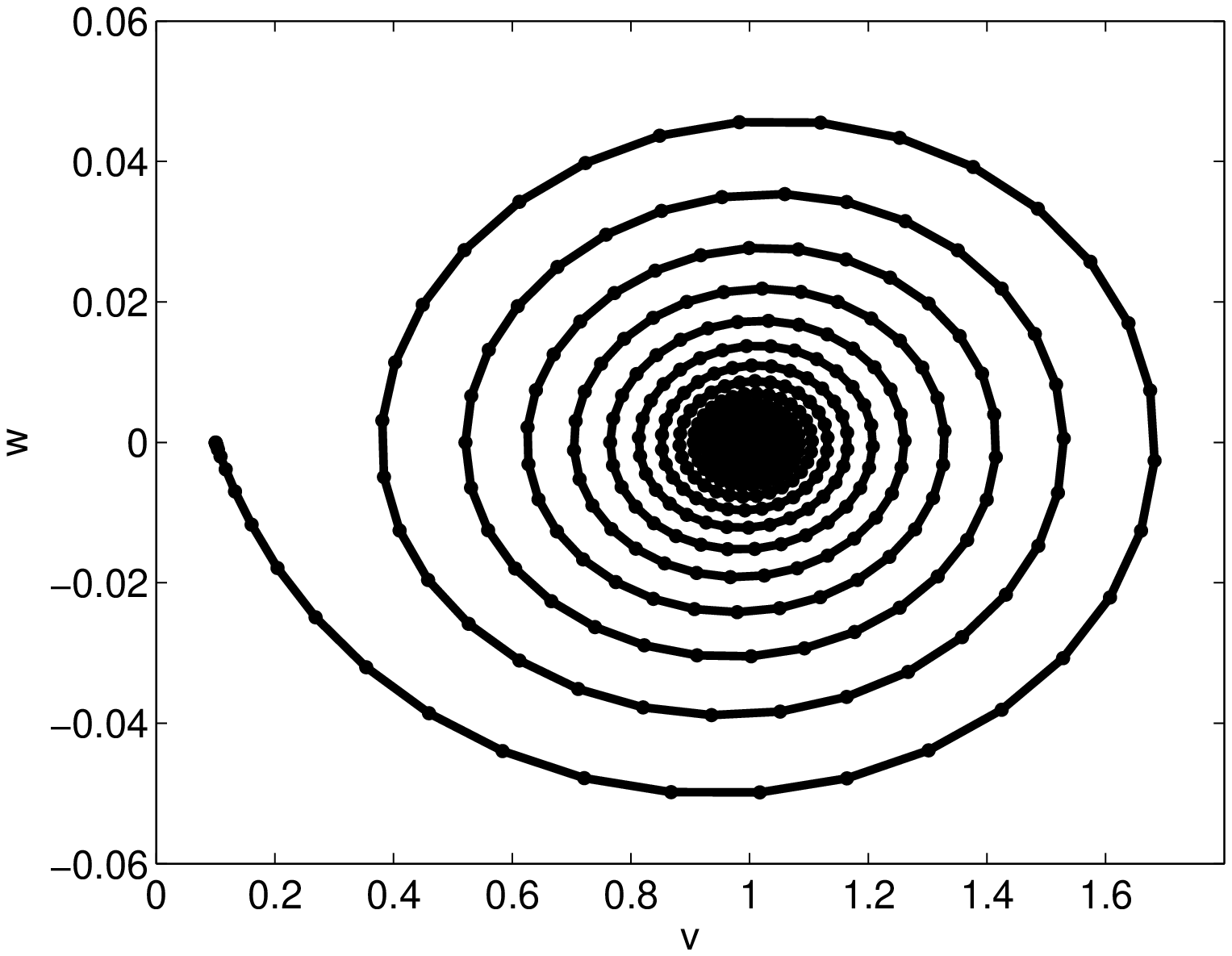} \\
\multicolumn{2}{c}{\mbox{\bf (c)}}
\end{array}$
\end{center}
\caption{Images of the profiles and their derivatives (left) and corresponding phase portraits (right)  for an adiabatic monatomic gas ($\gamma=5/3$) with $v_+=0.1$ and $d$ varying (note $d_*\approx 0.259$).  We demonstrate $(a)$ monotone profiles with $d=0.2$, $(b)$ non-monotone profiles which are mildly oscillatory with $d=2$, and $(c)$ non-monotone profiles that are highly oscillatory with $d=200$.
}
\label{profiles}
\end{figure}

\subsection{Classification of Profiles}

We show that smooth shock profiles are monotone for small values of $d$ and transition to highly oscillatory fronts when $d$ gets large; see Figure \ref{profiles} for illustrative examples.  The transition point between monotone and non-monotone profiles is found to be
\begin{equation}
\label{dstar}
d_* = \frac{1}{4v_-^2(1+a p'(v_-))},
\end{equation}
and in the case of an adiabatic gas with $v_-=1$, see Section \ref{adiabatic}, this becomes
\begin{equation}
\label{dstar2}
d_* = \frac{1}{4(1-a\gamma)} =  \frac{M^2}{4(M^2-1)} ,
\end{equation}
where $M$ is the Mach number.  In particular as the amplitude approaches zero, we have that $1+a p'(v_-) \rightarrow 0$, see \eqref{jxexp1}, thus making all profiles monotone regardless of $d$; this is consistent with the results in Section \ref{smallamp}.  In the large-amplitude limit, we have that $a\rightarrow 0$ and thus $d_*\rightarrow 1/(4 v_-^2)$.  Hence, for values of $d$ less than $1/(4 v_-^2)$, all profiles are monotone, regardless of amplitude, and for $d\geq 1/(4 v_-^2)$ a transition from monotone to non-monotone occurs for moderate to large amplitude fronts.  We have the following:

\begin{theorem}
Shock profiles of \eqref{slemrod} are monotone iff $0\leq d\leq d_*$.
\end{theorem}

\begin{proof}
By a geometric singular perturbation argument very similar to the one in Section \ref{smallamp}, we know that profiles are monotone for sufficiently small values of $d$.  When $d=d_*$, we can show that the local behavior near the fixed point $(v_-,0)$ transitions from that of an unstable node to an unstable spiral, which is clearly non-monotone.  Hence, it suffices to show that the profile does not lose monotonicity until $d$ passes through $d_*$.  By linearizing \eqref{firstorder} about $v_-$, we get the system
\begin{equation}
\label{linearode}
\begin{pmatrix}v\\w\end{pmatrix}' = \begin{pmatrix}0 & 1\\\frac{-(1+a p'(v_-))}{d} & \frac{1}{d v_-}\end{pmatrix} \begin{pmatrix}v\\w\end{pmatrix}.
\end{equation}
If monotonicity is lost before $d$ gets to $d_*$, then for some $d_0<d_*$ the phase curve connects to $v_-$ vertically.  This would require the vector field near $v_-$ to admit a vector in the $w$ direction.  However, since
\[
\begin{pmatrix}0 & 1\\\frac{-(1+a p'(v_-))}{d_0} & \frac{1}{v_-}\end{pmatrix} \begin{pmatrix}0\\-w\end{pmatrix} = \begin{pmatrix}-w\\-\frac{w}{d_0 v_-}\end{pmatrix},
\]
this cannot happen.  Hence the profiles are monotone whenever $d<d_*$.
\end{proof}

\begin{figure}[ht]
\begin{center}
\includegraphics[width=12cm]{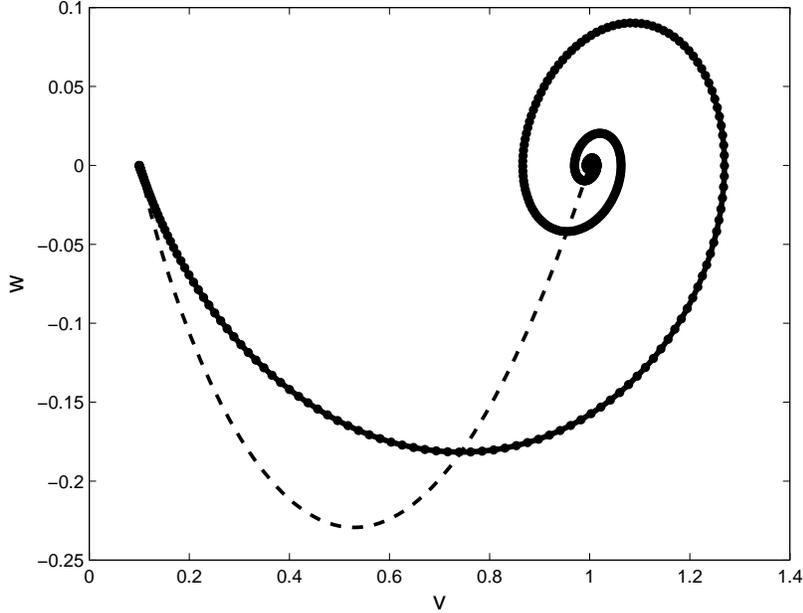}
\end{center}
\caption{Two phase portraits for an adiabatic monatomic gas ($\gamma=5/3$) for $v_+=0.1$.  The dark line corresponds to a non-monotone profile with $d=5$ and the dotted line is the zero-capilarity profile $w=\phi(v)$.  Note that the dotted curve intersects the dark one at its minimum.  Hence to bound the derivative of the non-monotone profile, we need only bound the derivative of the monotone profile.}
\label{vx}
\end{figure}

\subsection{Bounds on $|\bV_x|$}

W now provide bounds on $\bV_x$ that are used later in our analysis.  We show that $|\bV_x| \leq \varepsilon^2/4$, where $\varepsilon=v_--v_+$ is the amplitude of the profile.  This bound holds regardless of capillarity strength, and is important for our analysis in Section \ref{hfb}.  The idea behind the proof follows from Figure \ref{vx}, where we see that the maximum value of $|\bV_x|$ occurs at the point where the profile intersects the zero-capillarity profile.  Thus we need only find a bound on the zero-capillarity profile.

\begin{theorem}
\label{asympthm}
Shock profiles of \eqref{slemrod} satisfy $|\bV_x| \leq \varepsilon^2/4$, where $\varepsilon=|v_--v_+|$.  
\end{theorem}

\begin{proof}
Consider the phase portrait of the profile.  Let $v_0$ denote the point that maximizes $|\bV_x|$.  This occurs when $w'=0$ in \eqref{firstorder2}, or in other words, when $w = \phi(v)$, which is the zero-capillarity profile.  Hence, the maximum point for $|\bV_x|$ coincides with the zero-capillarity curve, which we can show is bounded above by $\varepsilon^2/4$.  This follows easily since
\[
\sup_{x\in\mathbb{R}} |\bV_x| = \sup_{v\in[v_+,v_-]} |\phi(v)| < \sup_{v\in[v_+,v_-]} |\bV(\bV-v_-)| \leq \frac{|v_--v_+|^2}{4} = \frac{\varepsilon^2}{4}.
\]
\end{proof}

\begin{remark}
In the $\varepsilon\rightarrow 0$ limit we can likewise show that $|\bV_{xx}|=|\bV_x|\mathcal{O}(\varepsilon)$.
\end{remark}

\subsection{Stability problem}

By linearizing \eqref{trav:cap2} about the profile $(\hat{v},\hat{u})$, we get the eigenvalue problem
\begin{equation}
\label{jx6}
\begin{split}
&\lambda v + v' - u' =0,\\
&\lambda u + u' - (f(\bV) v)' = \left(\frac{u'}{\bV}\right)' - d v''',
\end{split}
\end{equation}
where $f(\bV) = -a p'(\hat{v})-\bV_x/\bV^2$.  We say that a shock profile of \eqref{slemrod} is spectrally stable if the linearized system \eqref{jx6}  has no spectra in the closed deleted right half-plane given by $P = \{\R(\lambda) \geq 0\}\setminus\{0\}$, that is, there are no growth or oscillatory modes.  To show that the essential spectrum is stable, we linearize \eqref{trav:cap2} about the endstates $(v_\pm,u_\pm)$ and show that the resulting constant-coefficient system is stable; see \cite{He}.  This was done for general viscous-dispersive and higher-order systems in \cite{Hu.2}.  Thus it suffices to show that the point spectrum is also stable.  However, since traveling wave profiles always have a zero-eigenvalue due to translational invariance, it is often difficult to get good uniform bounds in energy estimates.  Hence, we use the standard technique of transforming into integrated coordinates; see \cite{Go.1, ZH, BHRZ}.  This goes as follows:

Suppose that $(v,u)$ is an eigenfunction of \eqref{jx6} with eigenvalue $\lambda\ne 0$.  Then
\[
\tilde{u}(x) = \int_{-\infty}^x u(z) dz , \quad
\tilde{v}(x) = \int_{-\infty}^x v(z) dz,
\]
and their derivatives decay exponentially as $x \rightarrow \infty$; see \cite{ZH}. Thus, by substituting and then integrating, $(\tilde{u},\tilde{v})$ satisfies (suppressing the tilde)
\begin{subequations}\label{ep}
\begin{align}
&\lambda v + v' - u' =0, \label{ep:1}\\
&\lambda u + u' - f(\bV) v' = \frac{u''}{\bV} - d v'''\label{ep:2}
\end{align}
\end{subequations}
This new eigenvalue problem is important because its point spectrum differs from that of \eqref{jx6} only at $\lambda=0$.  It follows that spectral stability of \eqref{jx6} is implied by spectral stability of \eqref{ep}.  Hence, we will use \eqref{ep} instead of \eqref{jx6} in the remainder of our stability analysis.

\section{Small-Amplitude Spectral Stability}
\label{spec}

In this section we show that small-amplitude smooth shock profiles are spectrally stable.  This work generalizes the energy methods in \cite{MN,BHRZ} to the case of an isentropic gas with capillarity.

\begin{theorem}
Small-amplitude shocks of \eqref{slemrod} are spectrally stable.
\end{theorem}

\begin{proof}

Suppose that $\R\lambda\geq 0$.  Recall that small-amplitude profiles are monotone with $\bV_x<0$ and thus also satisfy $f(\bV) > 0$ and $f'(\bV)<0$.  By multiplying \eqref{ep:2} by the conjugate $\bar{u}/f(\bV)$ and integrating in $x$ from $-\infty$ to $\infty$, we have
\[
\ip \frac{\lambda u \bar{u}}{f(\bV)} + \ip \frac{u' \bar{u}}{f(\bV)} -  \ip v' \bar{u} = \ip
\frac{u''\bar{u}}{\bV f(\bV)} - \ip \frac{d v''' \bar{u}}{f(\bV)}.
\]
Integrating the last three terms by parts and appropriately using \eqref{ep:1} to substitute for $u'$ in the third term gives us
\begin{align*}
&\ip \frac{\lambda |u|^2}{f(\bV)} + \ip \left[ \frac{1}{f(\bV)} + \left(\frac{1}{\bV f(\bV)}\right)' \right] u' \bar{u} + \ip v (\overline{\lambda v + v'}) + \ip \frac{|u'|^2}{\bV f(\bV)}\\
&\quad= d \ip \frac{1}{f(\bV)} v'' \bar{u}' + d\ip \left(\frac{1}{f(\bV)}\right)' v'' \bar{u}.
\end{align*}
We take the real part and appropriately integrate by parts:
\begin{align*}
&\R(\lambda)\ip \left[ \frac{|u|^2}{f(\bV)}+|v|^2 \right] + \ip g(\bV) |u|^2 + \ip \frac{|u'|^2}{\bV f(\bV)} \\
&\qquad= d\:\R\left[\ip \frac{1}{f(\bV)} v'' \bar{u}'+ \ip \left(\frac{1}{f(\bV)}\right)'v'' \bar{u}\right],
\end{align*}
where
\begin{equation}
\label{g}
g(\bV)= - \frac{1}{2} \left[ \left(\frac{1}{f(\bV)}\right)' + \left(\frac{1}{\bV f(\bV)}\right)'' \right].
\end{equation}
Thus, by integrating the last two terms by parts and further simplifying, for $\lambda \geq 0$, we have
\begin{equation}
\begin{split}
&  \ip g(\bV) |u|^2  + \ip \frac{|u'|^2}{\bV f(\bV)} - \frac{d}{2}\ip \left(\frac{1}{f(\bV)}\right)' |v'|^2 \\
&\qquad \leq - d\:\R\left[ 2\ip \left(\frac{1}{f(\bV)}\right)'v' \bar{u}' + \ip \left( \frac{1}{f(\bV)}\right)'' v' \bar{u}\right].
\end{split}
\end{equation}
We note that since $d \geq 0$ and $\bV_x <0$, then all the terms on the left-hand side are non-negative.  Moreover, since $|\bV_x|=\CalO(\varepsilon^2)$ and $|\bV_{xx}| = |\bV_x| \CalO(\varepsilon)$, it follows that the right-hand side of the above equation is bounded above by
\begin{align*}
- 2 d \ip \left(\frac{1}{f(\bV)}\right)' |v'| |u'| + C d \ip \varepsilon |\bV_x| |v'||u|.
\end{align*}
Thus, by Young's inequality, we have
\begin{align*}
& \ip g(\bV) |u|^2  + \ip \frac{|u'|^2}{\bV f(\bV)} - \frac{d}{2}\ip \left(\frac{1}{f(\bV)}\right)' |v'|^2 \notag \\
&\qquad < - 2 d \ip \left(\frac{1}{f(\bV)}\right)' \left[ \frac{|v'|^2}{4 \eta_1} +  \eta_1 |u'|^2 \right] + C \ip \varepsilon |\bV_x| \left[ \frac{|v'|^2}{4 \eta_2} +  \eta_2 |u|^2 \right].
\end{align*}
We can see that for $\eta_1 >1$ and $\eta_2,\varepsilon$ sufficiently small, the left side dominates the right side, which is a contradiction.
\end{proof}


\section{Monotone large-amplitude shocks}
\label{noreal}

In this section, we show that monotone profiles have no unstable real spectrum.  Our proof follows from a novel energy estimate that generalizes that of \cite{BHRZ} to a general ideal gas law and the addition of a capillarity term.  This restricts the class of admissible bifurcations for monotone profiles to those of Hopf-type, where one or more conjugate pairs of eigenvalues cross the imaginary axis.

\begin{theorem}
Monotone shocks of \eqref{slemrod} have no unstable real spectrum.
\end{theorem}
\begin{proof}
Suppose that $\lambda\in [0,\infty)$.  Since profiles are monotone, we have that $\bV_x < 0$. We multiply \eqref{ep:2} by the conjugate $\bar{v}$ and integrate in $x$ from $-\infty$ to $\infty$.  This gives
\[
\ip \lambda u \bar{v} + \ip u' \bar{v} -  \ip f(\bV) v' \bar{v} = \ip \frac{u''\bar{v}}{\bV} - d \ip v''' \bar{v}.
\]
Notice that on the real line, $\bar{\lambda}=\lambda$.  Thus, we have
\[
\ip \bar{\lambda}u \bar{v} + \ip u' \bar{v} - \ip f(\bV) v' \bar{v} = \ip \frac{u''\bar{v}}{\bV} + d \ip v'' \bar{v}'.
\]
Using \eqref{ep:1} to substitute for $\overline{\lambda v}$ in the first term and for $u''$ in the last term, we get
\[
\ip u (\bar{u}'-\bar{v}') + \ip u' \bar{v} - \ip f(\bV) v' \bar{v} = \ip \frac{(\lambda v' + v'')\bar{v}}{\bV} + d \ip v'' \bar{v}'.
\]
Separating terms and simplifying gives
\[
\ip u \bar{u}' + 2 \ip u' \bar{v} -\ip f(\bV) v' \bar{v} = \lambda  \ip \frac{v'\bar{v}}{\bV} + \ip \frac{v'' \bar{v}}{\bV} + d \ip v'' \bar{v}'.
\]
We further simplify by substituting for $u'$ in the second term and integrating  the last terms by parts to give,
\[
\ip u \bar{u}' + 2 \ip (\lambda v + v') \bar{v} - \ip \left( f(\bV) + \frac{\bV_x}{\bV^2} + \frac{\lambda}{\bV} \right) v' \bar{v} + \ip  \frac{|v'|^2}{\bV} = d \ip v'' \bar{v}',
\]
which yields
\[
\ip u \bar{u}' + 2 \lambda \ip |v|^2 + \ip \left(2 +a p'(\bV) - \frac{\lambda}{\bV}\right) v' \bar{v} + \ip \frac{|v'|^2}{\bV} = d \ip v'' \bar{v}'.
\]
By taking the real part (recall that $\lambda \in [0,\infty)$), we arrive at
\[
\lambda \ip \left(2 - \frac{\bV_x}{2\bV^2} \right) |v|^2 - \frac{a}{2} \ip  p''(\bV) \bV_x |v|^2 + \ip \frac{|v'|^2}{\bV} = 0.
\]
This is a contradiction.  Thus, there are no positive real eigenvalues for monotone shock layers in \eqref{slemrod} .
\end{proof}



\section{High-frequency bounds}
\label{hfb}

In this section, we prove high-frequency spectral bounds for monotone large-amplitude smooth shock profiles.  This provides a ceiling as to how far along both the imaginary and real axes that one must explore for point spectra when doing Evans function computations.  Indeed to check for roots of the Evans function in the unstable half-plane, say using the argument principle, one needs only compute within these bounds.  If no roots are found therein, then we have a numerical verification of spectral stability.  We remark that in this section and the next, we depart from the generality of an ideal gas, and restrict ourselves to the adiabatic case; see Section \ref{adiabatic}.  We remark, however, that we could have carried out our analysis for an ideal gas as long as $v_-=1$, which we can achieve by rescaling.  We have the following lemmata:

\begin{lemma}
The following identity holds for $\varepsilon_1,\varepsilon_2,\theta>0$ and $\R \lambda \geq 0$:
\begin{align}
(\R(\lambda) &+ |\I (\lambda)|)  \ip \bV |u|^2 + (1-\varepsilon_1 - \varepsilon_2) \ip |u'|^2\notag\\
& \leq \left[ \frac{1}{4\varepsilon_1} + \frac{C}{2 \theta}\right] \int \bV |u|^2  + d^2 \int \left[\frac{1}{4} + \frac{1}{2\varepsilon_2}\right] |v''|^2 + \theta \int f(\bV) |v'|^2  \label{id1},
\end{align}
where $C = \sup |f(\bV)\bV|$.
\end{lemma}

\begin{proof}
We multiply \eqref{ep:2} by $\bV {\bar u}$ and integrate along $x$ from $-\infty$ to $\infty$.  This yields
\[
\lambda \ip \bV |u|^2 + \ip \bV u'\bar{u} + \ip |u'|^2 = \ip f(\bV) \bV v'\bar{u} + d\int \bV_x v'' \bar{u} + d\int \bV v'' \bar{u}' .
\]
Taking the real and imaginary parts, adding them together, and noting that $|\R(z)| + |\I(z)| \leq \sqrt{2}|z|$, yields
\begin{align*}
&(\R(\lambda) + |\I (\lambda)|)  \ip \bV |u|^2 -\frac{1}{2}\ip \bV_x|u|^2 + \ip |u'|^2\\
& \leq \ip \bV |u||u'| + \sqrt{2} \ip f(\bV)\bV|v'||u| + \sqrt{2} d \left[ \ip |\bV_x| |v''| |u| + \ip \bV |v''| |u'|\right] \\
& \leq \varepsilon_1 \ip \bV |u'|^2 + \frac{1}{4\varepsilon_1}\ip \bV |u|^2 +\theta \ip f(\bV) |v'|^2 + \frac{1}{2\theta}\ip f(\bV) \bV^2|u|^2 \\
&\qquad + \frac{1}{2} \ip |\bV_x| |u|^2 + d^2 \int |\bV_x| |v''|^2 + \varepsilon_2 \ip \bV |u'|^2 + \frac{d^2}{2 \varepsilon_2} \ip \bV |v''|^2\\
& < (\varepsilon_1+\varepsilon_2) \ip |u'|^2 + \left[ \frac{1}{4\varepsilon_1} + \frac{C}{2\theta}\right] \ip \bV |u|^2 +\theta \ip f(\bV) |v'|^2 \\
&\qquad + \frac{1}{2} \ip |\bV_x| |u|^2 + d^2 \int \left[\frac{1}{4} + \frac{1}{2\varepsilon_2}  \right]|v''|^2.
\end{align*}
Rearranging terms yields \eqref{id1}.
\end{proof}

\begin{lemma}
\label{kawashima}
The following identity holds for $\R \lambda \geq 0$:
\begin{equation}
\label{id2}
\ip |u'|^2 \geq \frac{1}{2} \ip \left[ f(\bV) - p'(\bV) \right] |v'|^2 + d \ip |v''|^2.
\end{equation}
\end{lemma}

\begin{proof}
We multiply \eqref{ep:2} by ${\bar v'}$ and integrate along $x$ from $-\infty$ to $\infty$.  This yields
\[
\lambda \ip u\bar{v}' + \ip u'\bar{v}' - \ip f(\bV) |v'|^2 = \ip \frac{1}{\bV}u''\bar{v}' - d\ip v'''\bar{v}'.
\]
Using \eqref{ep:1} on the right-hand side, integrating by parts, and taking the real part gives
\[
\R \left[ \lambda \ip u\bar{v}' + \ip u'\bar{v}'\right] = \ip \left[ f(\bV) + \frac{\bV_x}{2 \bV^2} \right] |v'|^2 + \R(\lambda)\ip \frac{|v'|^2}{\bV} + d\ip|v''|^2.
\]
In our domain of interest, this yields
\begin{equation}
\label{id3_1}
\R \left[ \lambda \ip u\bar{v}' + \ip u'\bar{v}'\right] \geq \frac{1}{2} \ip \left[f(\bV) - p'(\bV) \right] |v'|^2  + d\ip|v''|^2
\end{equation}
Now we manipulate the left-hand side.  Note that
\begin{align*}
\lambda \ip u\bar{v}' + \ip u'\bar{v}' &= (\lambda+\bar{\lambda}) \ip u\bar{v}' - \ip u(\bar{\lambda}\bar{v}' + \bar{v}'')\\
&= -2\R(\lambda) \ip u' \bar{v} - \ip u \bar{u}''\\
&= -2\R(\lambda) \ip (\lambda v + v') \bar{v} + \ip |u'|^2.
\end{align*}
Hence, by taking the real part we get
\[
\R \left[ \lambda \ip u\bar{v}' + \ip u'\bar{v}'\right] = \ip |u'|^2 - 2\R(\lambda)^2 \ip |v|^2.
\]
This combines with \eqref{id3_1} to give \eqref{id2}.
\end{proof}

Now we prove our high-frequency bounds.

\begin{theorem}
\label{thmhfb}
For a monotone profile with $0 \leq d \leq 1/3$, any eigenvalue $\lambda$ of \eqref{ep} with nonnegative real part satisfies
\begin{equation}
\label{hfbounds}
\R(\lambda) + |\I(\lambda)| \leq 3 + \frac{12 C}{5},
\end{equation}
where $C = \sup|f(\bV)\bV|$.
\end{theorem}

\begin{proof}
Combining \eqref{id1} and \eqref{id2}, we have
\begin{align*}
(\R(\lambda) &+ |\I (\lambda)|)  \ip \bV |u|^2 + (1-\varepsilon_1 - \varepsilon_2) \left[ \frac{1}{2}\ip f(\bV)|v'|^2 + d \ip |v''|^2 \right] \\
& \leq \left[ \frac{1}{4\varepsilon_1} + \frac{C}{2 \theta}\right] \int \bV |u|^2  + d^2 \int \left[\frac{1}{4} + \frac{1}{2\varepsilon_2}\right] |v''|^2 + \theta \int f(\bV) |v'|^2.
\end{align*}
Setting $\theta = (1-\varepsilon_1 - \varepsilon_2)/2$ yields
\begin{align*}
(\R(\lambda) &+ |\I (\lambda)|)  \ip \bV |u|^2 + (1-\varepsilon_1 - \varepsilon_2) d \ip |v''|^2 \\
& \leq \left[ \frac{1}{4\varepsilon_1} + \frac{C}{1-\varepsilon_1 - \varepsilon_2}\right] \int \bV |u|^2  + d^2 \int \left[\frac{1}{4} + \frac{1}{2\varepsilon_2}\right] |v''|^2.
\end{align*}
Hence for  $0 \leq d \leq 1/3$, choose $\varepsilon_1 = 1/12$ and $\varepsilon_2 = 1/2$ to get \eqref{hfbounds}.
\end{proof}

\begin{remark}
For an adiabatic gas, $p(\bV) = v^{-\gamma}$, $\gamma\geq 1$, we can show that $C \leq \gamma$; see \cite{BHRZ}.  Thus in the range $\gamma\in[1,3]$ we can safely bound the unstable spectrum with a half circle of radius 12.  This compactifies the region of admissible unstable spectrum, thus allowing us to numerically compute winding numbers of the Evans function and determine whether shock layers are spectrally stable.
\end{remark}


\section{Evans function computation}
\label{evans}

In this section, we numerically compute the Evans function to determine whether any unstable eigenvalues exist in our system.  The Evans function $D(\lambda)$ is analytic to the right of the essential spectrum and is defined as the Wronskian of decaying solutions of \eqref{ep}; see \cite{AGJ}.  In a spirit similar to the characteristic polynomial, we have that $D(\lambda)=0$ if and only if $\lambda$ is an eigenvalue of the linearized operator \eqref{ep}.  While the Evans function is generally too complex to compute analytically, it can readily be computed numerically; see \cite{HuZ.2} and references within.

Since the Evans function is analytic in the region of interest, we can numerically compute its winding number in the right-half plane.  This allows us to systematically locate roots (and hence unstable eigenvalues) within.  As a result, spectral stability can be determined, and in the case of instability, one can produce bifurcation diagrams to illustrate and observe its onset.  This approach was first used by Evans and Feroe \cite{EF} and has been applied to various systems since; see for example \cite{PSW,AS,Br.2,BDG}.

\subsection{Numerical Setup}

We begin by writing \eqref{ep} as a first-order system $W' = A(x,\lambda) W$, where
\begin{equation}
\label{evans_ode}
A(x,\lambda) = \begin{pmatrix}0 & \lambda & 1 & 0\\0 & 0 & 1 & 0\\0 & 0 & 0 & 1\\ \lambda/d&\lambda/d& h/d & -(d\bV)^{-1}\end{pmatrix},\quad W = \begin{pmatrix} u\\v\\v'\\v''\end{pmatrix},
\end{equation}
and $h = h(\bV,\lambda) := 1+a p'(\bV)+\bV_x/\bV^2-\lambda/\bV$.  Note that eigenvalues of \eqref{ep} correspond to nontrivial solutions of $W(x)$ for which the boundary conditions $W(\pm\infty)=0$ are satisfied.  We remark that since $\bV$ is asymptotically constant in $x$, then so is $A(x,\lambda)$.  Thus at each end-state, we have the constant-coefficient system
\begin{equation}
\label{apm}
W' = A_\pm(\lambda) W, \qquad A_\pm(\lambda) := \lim_{x\rightarrow\pm\infty}A(x,\lambda).
\end{equation}
Hence solutions that satisfy the needed boundary condition must emerge from the 2-dimensional unstable manifold $W_1^-(x) \wedge W_2^-(x)$ at $x=-\infty$ and also the 2-dimensional stable manifold $W_3^+(x) \wedge W_4^+(x)$ at $x=\infty$.  In other words, eigenvalues of \eqref{ep} correspond to the values of $\lambda$ for which these two manifolds intersect, or more precisely, when $D(\lambda) = 0$, where
\[
D(\lambda) := (W_1^- \wedge W_2^- \wedge W_3^+ \wedge W_4^+)_{\mid x=0} = \det(W_1^- W_2^- W_3^+ W_4^+)_{\mid x=0}.
\]

We cannot naively produce the stable and unstable manifolds numerically.  Indeed with two exponential growth and decay modes, problems with stiffness arise.  Hence, we use the compound-matrix method to analytically track the stable and unstable manifolds; see \cite{AlBr,BDG, Br.1, Br.2,HuZ.2}.  Specifically we lift $A(x,\lambda)$ into the exterior-product space $\Lambda^2(\mathbb{C}^4) \approx \mathbb{C}^6$ to get
\[
A^{(2)}(x,\lambda) = \begin{pmatrix}
0&1&0&-1&0&0\\
0&0&1&\lambda&0&0\\
\lambda/d &h/d&-(d\bV)^{-1}&0&\lambda&1\\
0&0&0&0&1&0\\
-\lambda/d&0&0&h/d&-(d\bV)^{-1}&1\\
0&-\lambda/d&0&-\lambda/d&0 & -(d\bV)^{-1}.
\end{pmatrix}
\]
We then consider single trajectories $W_\pm(x)$ of the ``lifted'' problem
\[
W' = A^{(2)}(x,\lambda) W
\]
on each side corresponding to the simple dominant growth and decay modes at the left and right end states, respectively.  These trajectories correspond to the 2-forms  $W_1^-(x) \wedge W_2^-(x)$ and $W_3^+(x) \wedge W_4^+(x)$, and can be effectively wedged together when they meet at zero; see \cite{AlBr} for an excellent overview of this method.  

As an alternative, we consider the adjoint formulation of the Evans function \cite{PW,BSZ}.  Specifically, we integrate the trajectory $\widetilde{W}_+$ along the largest growth mode of the adjoint ODE
\begin{equation}
\label{adjode}
\widetilde{W}' = -A^{(2)}(x,\lambda)^* \widetilde{W}.
\end{equation}
starting at $x=\infty$.  We then define the (adjoint) Evans function to be $D_+(\lambda): = (\widetilde{W}_+\cdot  W_-)_{\mid x=0}$.  Note that $\widetilde{W}_+$ corresponds to the orthogonal complement of the 2-form $W_3^+(x) \wedge W_4^+(x)$ and so orthogonality of $\widetilde{W}_+$ and $W_-$ corresponds to intersection of the stable and unstable manifolds.

To further improve the numerical efficiency and accuracy of the shooting scheme, we rescale $W$ and $\widetilde{W}$ to remove exponential growth/decay at infinity, and thus eliminate potential problems with stiffness.  Specifically, we let $W(x) = e^{\mu^- x} V(x)$, where $\mu^-$ is the largest growth rate of the unstable manifold at $x=-\infty$, and we solve instead $V'(x) = (A^{(2)}(x,\lambda)-\mu^- I)V(x)$.  We initialize $V(x)$ at $x=-\infty$ as eigenvector $r^-$ of $A^{(2)}_-(\lambda)$ corresponding to $\mu^-$.  Similarly, it is straightforward to rescale and initialize $\widetilde{W}(x)$ at $x=\infty$.  This method is known to have excellent accuracy \cite{Br.1,Br.2,BrZ,BDG,HuZ.2,BHRZ};  in addition, the adaptive refinement gives automatic error control.  Finally, in order to maintain analyticity, the initial eigenvectors $r^-(\lambda)$ are chosen analytically using Kato's method; see \cite[pg. 99]{Kato} and also \cite{BrZ,BDG,HSZ}.

\begin{figure}[t]
\begin{center}
$\begin{array}{cc}
\includegraphics[width=6cm]{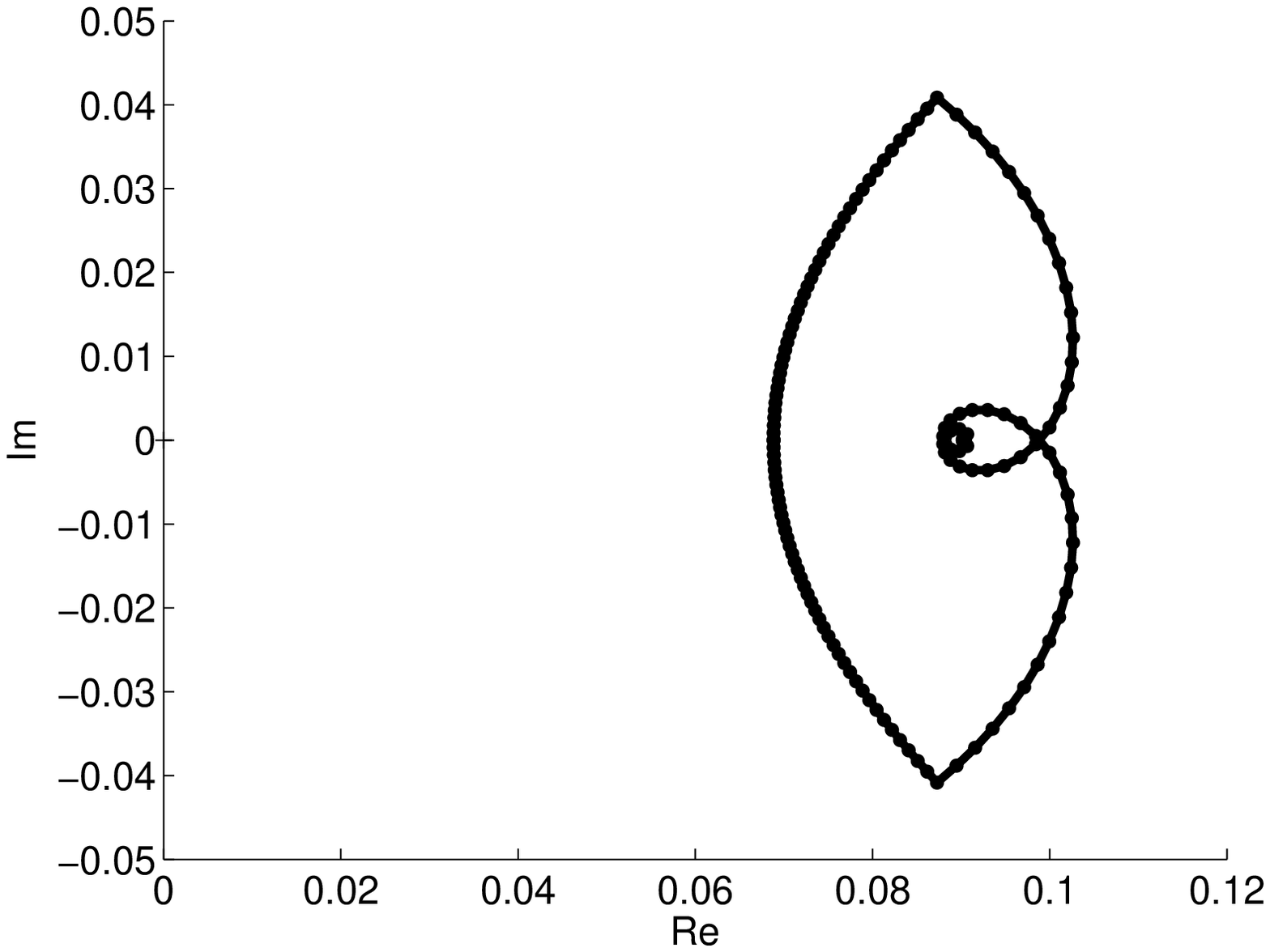} & \includegraphics[width=6cm]{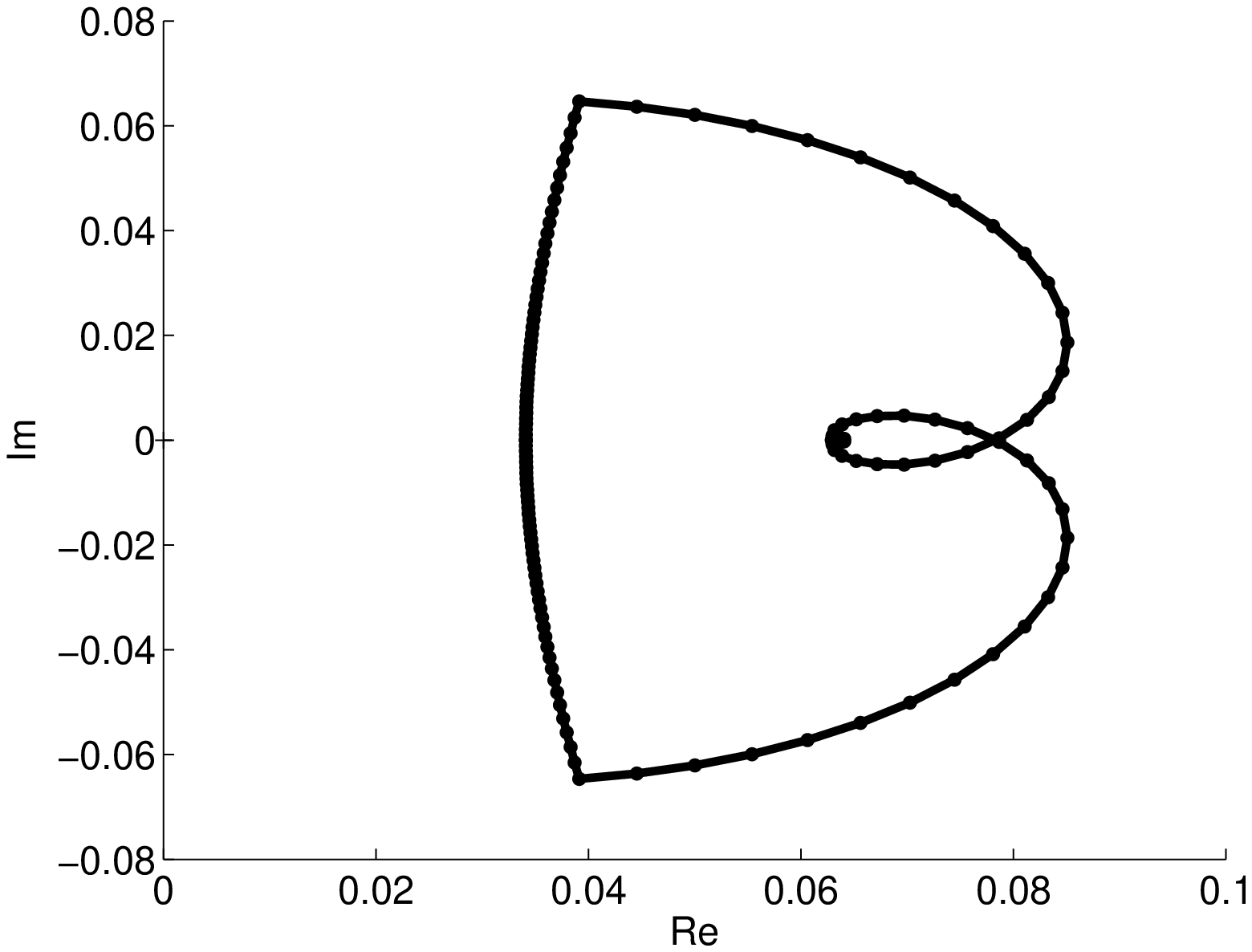} \\
\mbox{\bf (a)} & \mbox{\bf (b)}\\
\includegraphics[width=6cm]{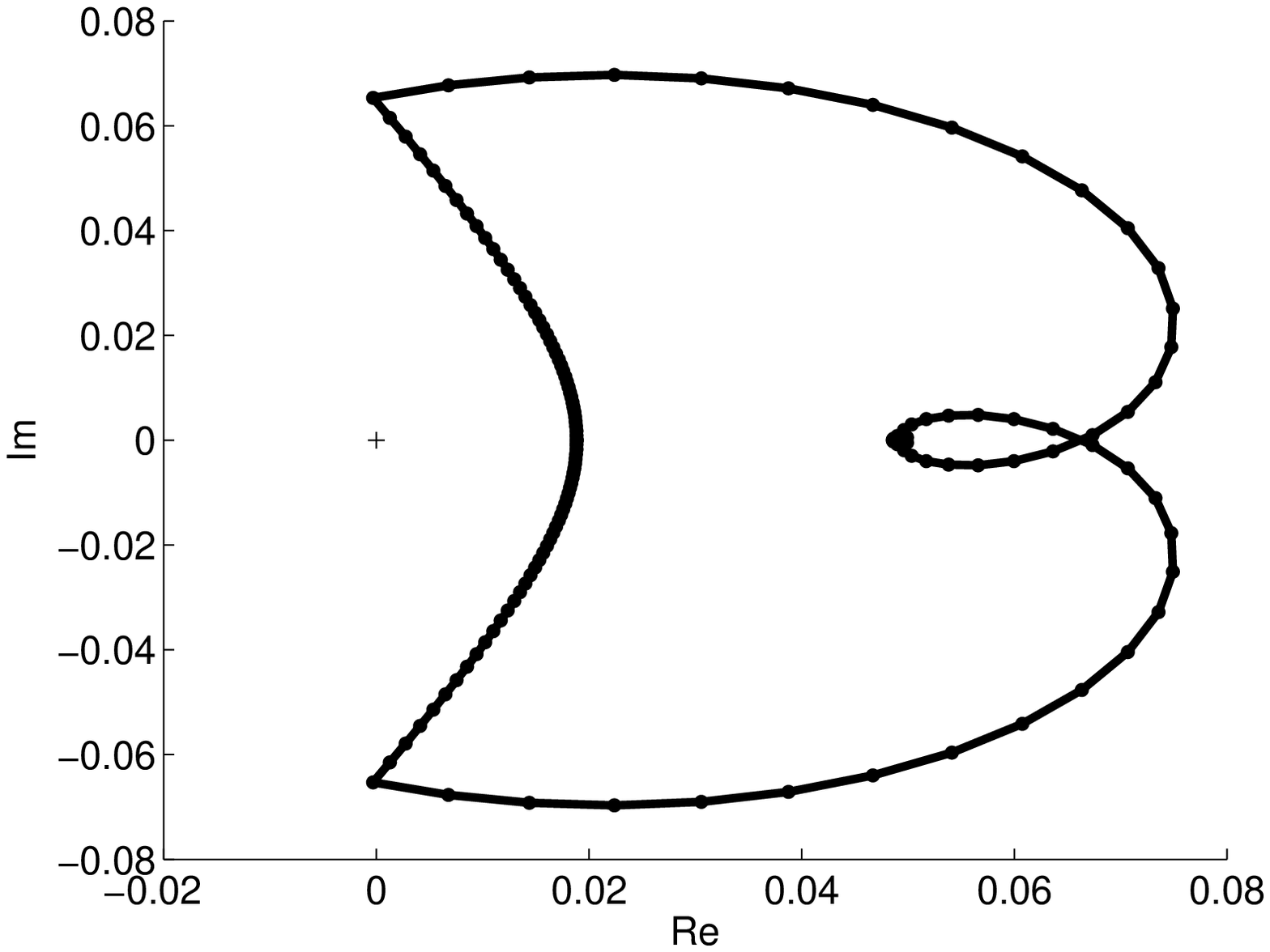} & \includegraphics[width=6cm]{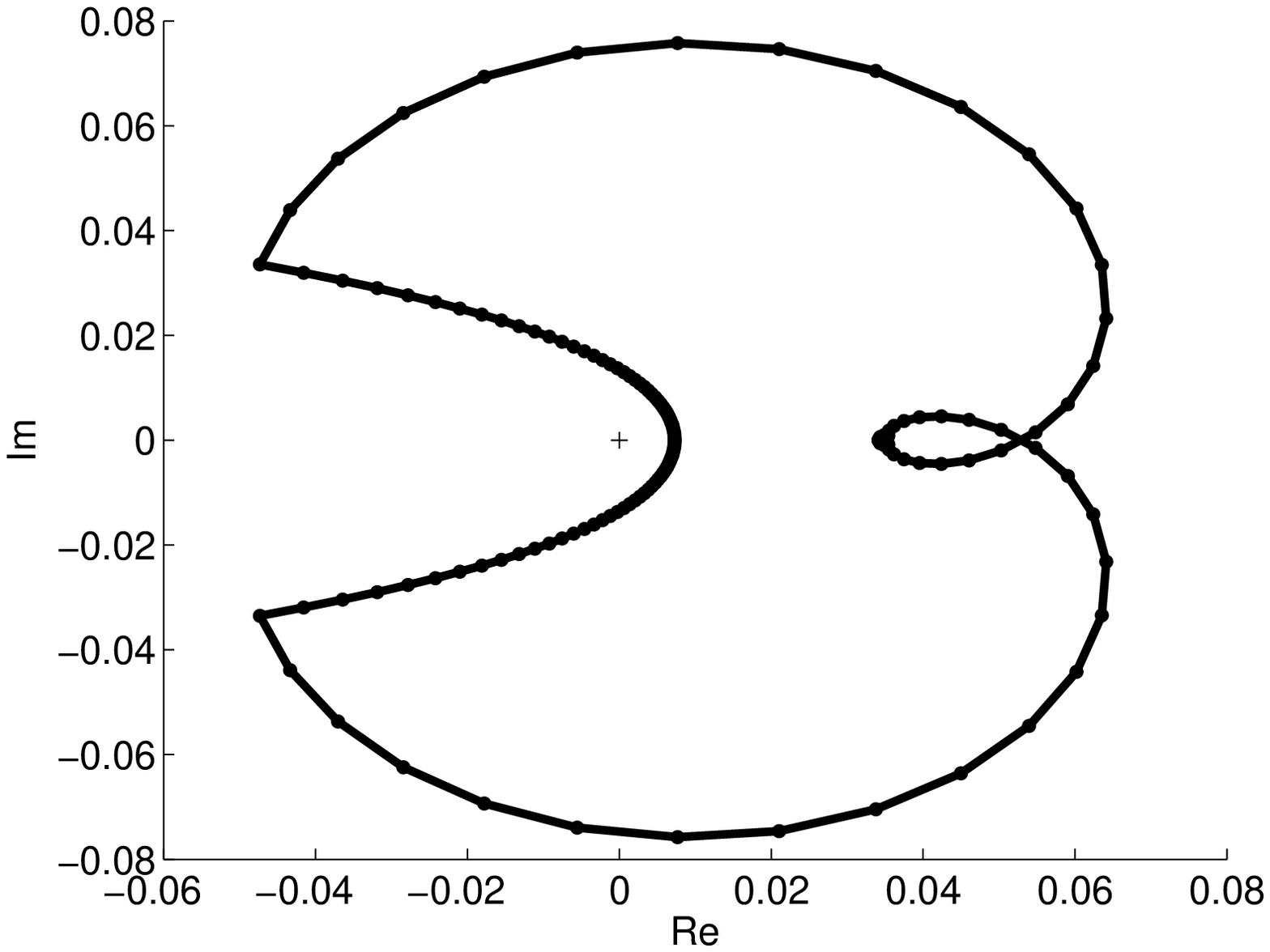} \\
\mbox{\bf (c)} & \mbox{\bf (d)}\\
\includegraphics[width=6cm]{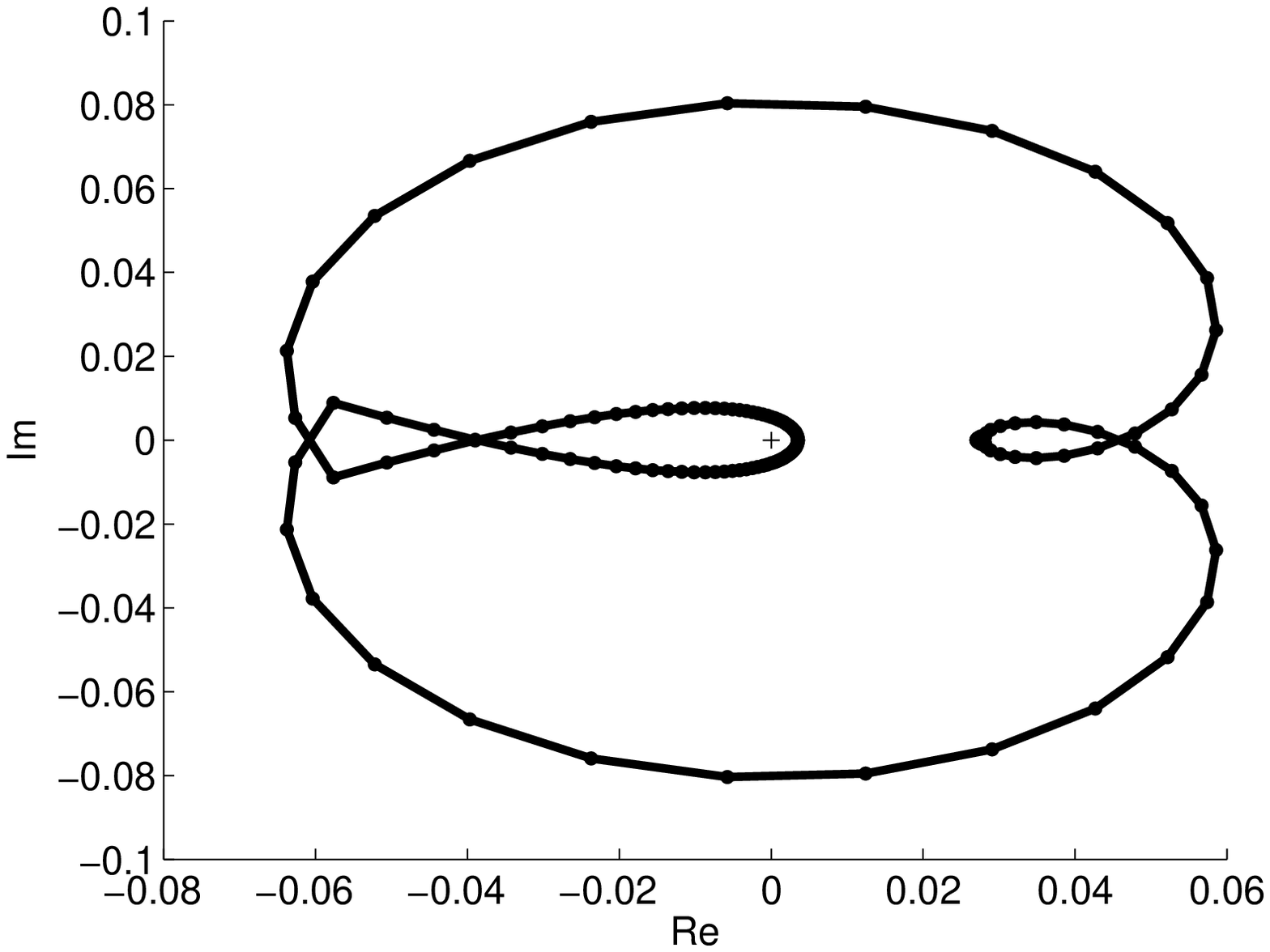} & \includegraphics[width=6cm]{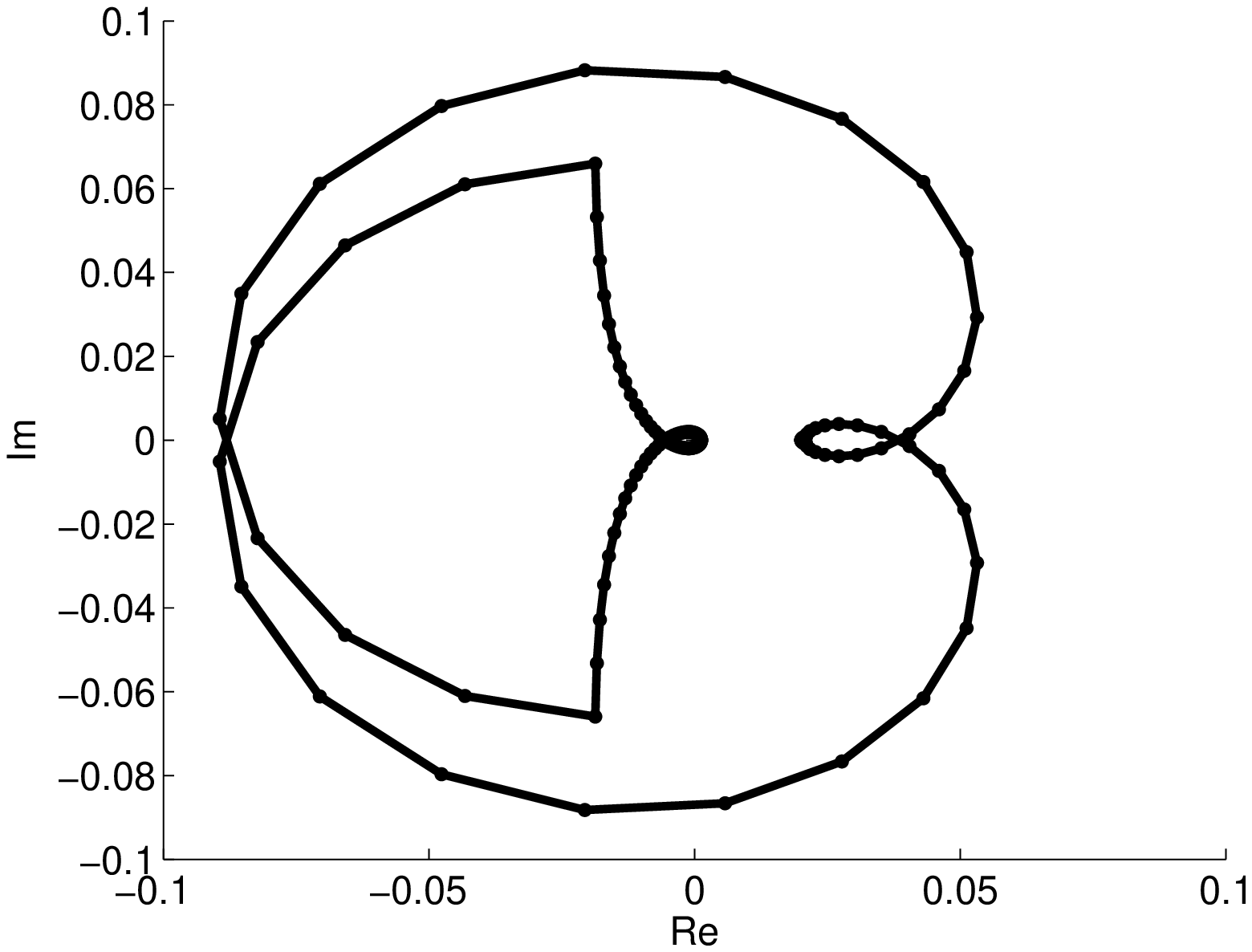} \\
\mbox{\bf (e)} & \mbox{\bf (f)}\\
\end{array}$
\end{center}
\caption{Evans function output for semi-circular contour of radius 12 with $d=0.45$ and (a) $v_+=0.65$, (b) $v_+=0.45$, (c) $v_+=0.35$, (d) $v_+=0.25$, (e) $v_+=0.20$, and (f) $v_+=0.15$.  Although the contours wrap around the origin as the shock strength increases, they clearly have winding number zero, thus demonstrating spectral stability.}
\label{runs}
\end{figure}

\subsection{Numerical Experiments}

We truncate the domain to a sufficiently large interval $[L_-,L_+]$ in order to do numerical computation.  Some care needs to be taken, however, to make sure that we go out far enough to produce good results.  Our experiments, described below, were primarily conducted using $L_\pm=\pm 25$, but for weaker shocks we had to go out as far as $L_\pm=\pm 50$.  For highly oscillatory profiles, very large values of $L_\pm$ are needed because the (under-damped) decay rate can be small.  To compute the profile, we used Matlab's {\tt bvp4c} routine, which is an adaptive Lobatto quadrature scheme.

\begin{figure}[t]
\begin{center}
\includegraphics[width=12cm]{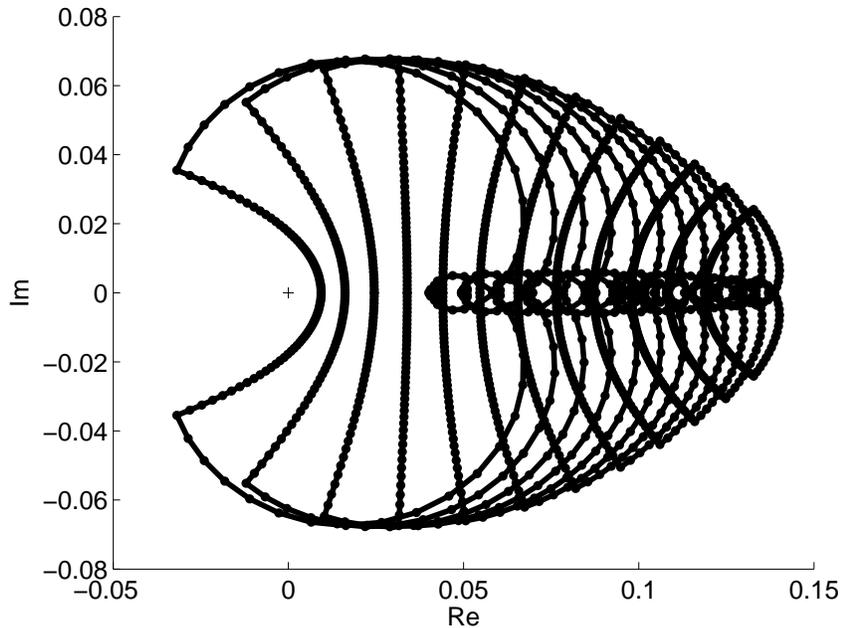} 
\end{center}
\caption{Evans function output of a semi-circular contour with $d=0.75$ and $v_+\in [0.20,0.80]$.  As the shock strength increases, the contours get closer to the origin and begin to wrap around it.  In the small shock limit, the contour drifts away from the origin and gets smaller.}
\label{several}
\end{figure}

\begin{figure}[t]
\begin{center}
\includegraphics[width=12cm]{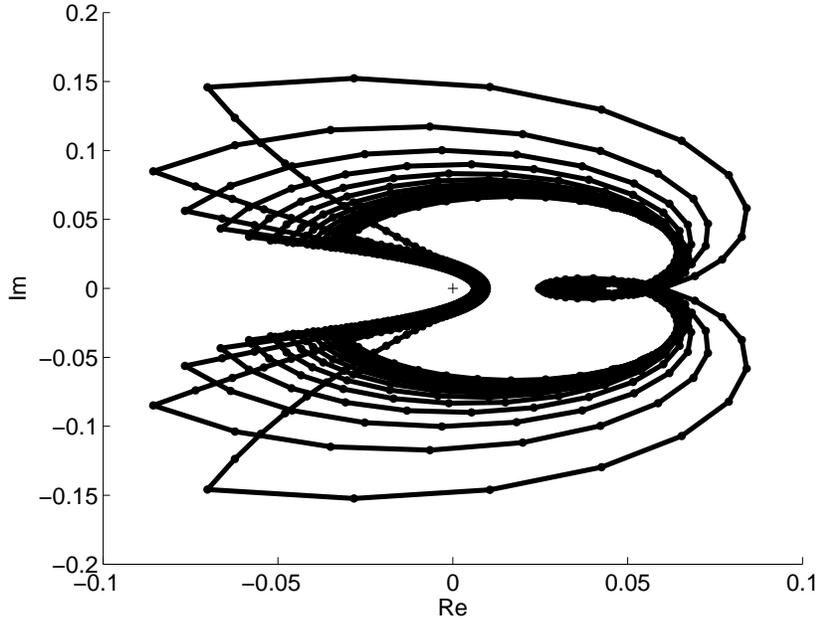} 
\end{center}
\caption{Evans function output of a semi-circular contour with $v_+=0.25$ and $d\in [0.15,0.80]$. As the $d$ decreases, the contours get larger and more spread out. }
\label{several2}
\end{figure}

\begin{figure}[t]
\begin{center}
\includegraphics[width=12cm]{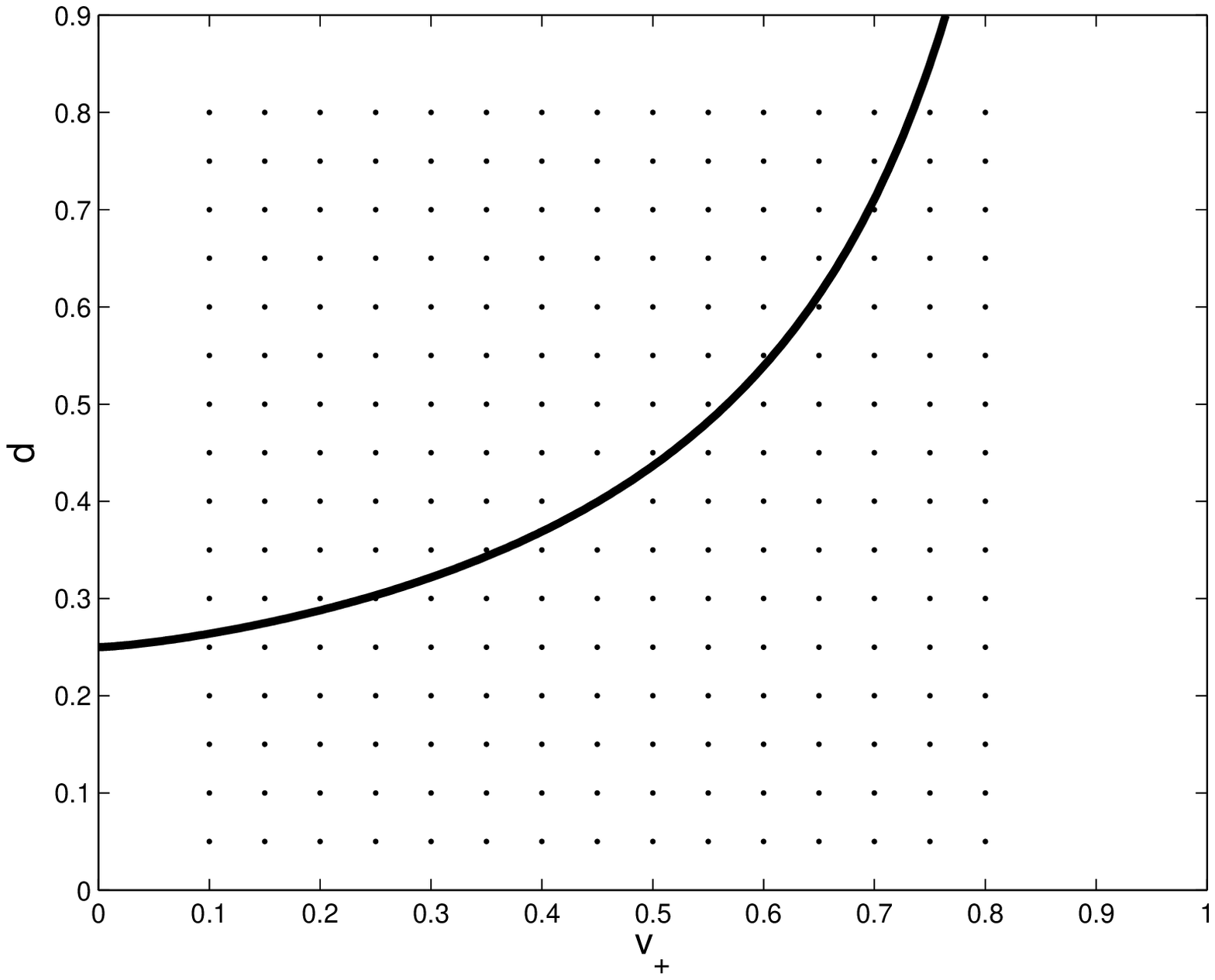} 
\end{center}
\caption{Dots correspond to runs with parameters $(v_+,d)$.  The upward increasing curve corresponds to the critical value $d_*$ between monotone and oscillatory shock profiles.}
\label{param}
\end{figure}

Our experiments were carried out uniformly on the range
\[
(v_+, d)\in [0.10,0.80]\times [0.05,0.80],
\]
with $\gamma=1.4$ (diatomic gas).  In terms of Mach number, this corresponds roughly to $1.15 \leq M \leq 5$, which covers the supersonic range and goes into the hypersonic regime.  Indeed $M\approx 5$ may even go beyond the physical range of the model.  For each $(v_+, d)$ on our grid, we computed the Evans function along a semi-circular contour in the right-half plane of radius $12$ centered at the origin.  Recall that for $d\leq 1/3$, this contains the admissible region of unstable spectrum from our high-frequency bounds.  The ODE calculations for individual values of $\lambda$ were carried out using Matlab's {\tt ode45} routine, which is the adaptive 4th-order Runge-Kutta-Fehlberg method (RKF45).  Typical runs involved between $100$ and $700$ mesh points, with error tolerance set to {\tt AbsTol = 1e-6} and {\tt RelTol = 1e-8}.  Values of $\lambda$ were varied on the semi-circular contour with $70$ points in the first quadrant, 40 on the arc and 30 along the imaginary axis, and then reflected along the real axis due to the conjugate symmetry of the Evans function, that is, $\overline{D(\lambda)}=D(\overline{\lambda})$.

In Figure \ref{runs}, we see a typical run for increasing $v_+$.  Notice that the contour wraps around the origin as the shock strength increases.  Thus it is difficult to conclude stability in the strong shock limit; this is a topic for future consideration.  Note that the graph gets farther away from the origin as the shock strength decreases, thus strongly suggesting stability in the small-amplitude limit.  In Figure \ref{several}, we see this effect more clearly.  In Figure \ref{several2}, we hold the shock strength fixed and vary $d$.  As $d$ approaches zero, we see the contour getting larger and more spread out.  Otherwise output does not seem to vary much in $d$, at least in our region of interest.

The actual parameter values computed were 
\[
(v_+,d) \in \{0.10,0.15,\ldots,0.80\} \times \{ 0.05,0.10,\ldots,0.80\};
\]
see Figure \ref{param}.  In total $240$ runs were conducted, all of which had winding number zero.  This effectively demonstrates spectral stability for monotone and nearly monotone profiles with $d\leq 1/3$ and strongly suggests spectral stability elsewhere in our region of study. Indeed the output is strikingly similar throughout.  Nonetheless, for $d>>1$ our profile becomes highly oscillatory and so it is not unreasonable to expect an instability to occur in the extreme.  This is a good direction for future work.

\section{Discussion and Open Problems}

We note that \eqref{evans_ode} blows up as $v_+\rightarrow 0$, and moreover the eigenvalues get far apart, thus  causing extreme stiffness.  Hence we have numerical difficulties for strong shocks, e.g, $M>>5$.  Difficulties also arise for both large and small values of $d$.  In particular the profile becomes highly oscillatory and numerically intractable for very large values of $d$, and as $d\rightarrow 0$ we likewise have that \eqref{evans_ode} blows up.  Nonetheless, we may be able to demonstrate stability as $d\rightarrow 0$ analytically as a singular limit of the $d=0$ case, which is stable; see \cite{HLZ}.

Slemrod's model is an ideal system for further investigation.  Not only is it physically relevant, and in some sense a canonical viscous-dispersive system, but it also pushes the boundaries of current numerical methods.  While this model has nice features such as monotone profiles, it also has highly complex and numerically taxing obstacles such as highly oscillatory profiles and large spectral separation between modes of \eqref{evans_ode} in the extreme parameter regime.  We intend to study this model further.

\def\ocirc#1{\ifmmode\setbox0=\hbox{$#1$}\dimen0=\ht0 \advance\dimen0
  by1pt\rlap{\hbox to\wd0{\hss\raise\dimen0
  \hbox{\hskip.2em$\scriptscriptstyle\circ$}\hss}}#1\else {\accent"17 #1}\fi}
  \def\cprime{$'$}

\end{document}